\documentclass{birkjour}
\usepackage{amssymb}
 \newtheorem{theorem}{Theorem}[section]
 \newtheorem{corollary}[theorem]{Corollary}
 \newtheorem{lemma}[theorem]{Lemma}
 
\newtheorem{proposition}[theorem]{Proposition}
 \theoremstyle{definition}
 
 \theoremstyle{remark}
 \newtheorem{remark}[theorem]{Remark}
 
 \numberwithin{equation}{section}

\renewcommand{\phi}{\varphi}

\newcommand{\co}{\mathbb{C}}
\newcommand{\CC}{\mathbb{C}}
\newcommand{\N}{\mathbb{N}}
\newcommand{\Z}{\mathbb{Z}}

\newcommand{\pp}{\mathcal{P}}

\newcommand{\cp}{\mathbb{C_+}}
\newcommand{\cm}{\mathbb{C_-}}
\newcommand{\R}{\mathbb{R}}
\newcommand{\rl}{\mathbb{R}}

\newcommand{\ima}{{\rm Im}\,}

\newcommand{\A}{\mathcal{A}}
\newcommand{\LL}{\mathcal{L}}
\newcommand{\hh}{\mathcal{H}}
\newcommand{\he}{\mathcal{H}(E)}
\newcommand{\ZZ}{\mathcal{Z}}
\newcommand{\hht}{\hh(T,A,\mu)}
\newcommand{\htt}{\hh(T,A,\mu)}

\newcommand{\nn}{\mathcal{N}}

\newcommand{\tg}{{\rm tg}\,}

\begin{document}

%
%
%
%
%
%
%
%
%

\title[]{Cauchy--de Branges spaces, geometry of their reproducing kernels 
and multiplication operators}

\author[Anton Baranov]{Anton Baranov}

\address{%
Department of Mathematics and Mechanics\\ 
St. Petersburg State University\\
St. Petersburg, Russia }

\email{anton.d.baranov@gmail.com}

\thanks{The results of Sections 3, 4, 5 and 6 were 
obtained with the support of the Russian Science Foundation grant 
19-71-30002. 
The results of Sections 7 and 8 were obtained with the support 
of Ministry of Science and Higher Education of the Russian Federation, 
agreement No 075-15-2021-602, and of 
Russian Foundation for Basic Research grant 20-51-14001-ANF-a.}




\begin{abstract}
Cauchy--de Branges spaces are Hilbert spaces of entire functions defined in terms of
Cauchy transforms of discrete measures on the plane and generalizing the classical de Branges theory. 
We consider extensions of two important properties of de Branges spaces to this, more general, setting.
First, we discuss geometric properties (completeness, Riesz bases) of systems of reproducing kernels
corresponding to the zeros of certain entire functions associated to the space. In the case of de Branges spaces they correspond to
orthogonal bases of reproducing kernels. The second theme of the paper is a characterization of the density of 
the domain of multiplication by $z$ in Cauchy--de Branges spaces. 
\end{abstract}

\maketitle
\sloppy

\section{Introduction}

Hilbert spaces of entire functions play an important part in modern analysis. Their structural properties 
(e.g., problems of uniqueness, interpolation or sampling) are of interest both from the function-theoretic point of view
and for numerous applications -- spectral problems for canonical systems (de Branges spaces), 
signal processing and time--frequency analysis (Paley--Wiener space, Bargmann--Segal--Fock space), 
theoretical physics, etc.

Recently, a systematic study of a new class of Hilbert spaces of entire functions was initiated 
by Yu.~Belov, T.~Mengestie and K.~Seip in \cite{bms1, bms2} and by E.~Abakumov, Yu.~Belov and the author 
in \cite{abb19, loc2, bar18}. These spaces, defined via discrete Cauchy transform and generalizing de Branges spaces, 
were named the {\it Cauchy--de Branges spaces} in \cite{abb19}. One of motivations for the study 
of the Cauchy--de Branges spaces is a functional model for rank one perturbations of compact normal operators 
\cite{bar18} (see Section \ref{appli}).

Let $T = \{t_n\}_{n=1}^\infty \in \co$ be a set of distinct complex numbers such that $|t_n| \to \infty$, $n\to \infty$,
and let $\mu = \sum_n \mu_n \delta_{t_n}$ with $\sum_n \frac{\mu_n}{|t_n|^2+1} <\infty$. 
We keep this notation throughout the whole paper.
With the pair $(T, \mu)$ we associate the space of Cauchy transforms
$$
\hh(T,\mu) = \bigg\{f(z) = \sum_n \frac{c_n\mu_n^{1/2}}{z-t_n}: \ (c_n)\in \ell^2 \bigg\}
$$
equipped with the norm $\|f\|_{\hh(T,\mu)} = \|(c_n)\|_{\ell^2}$. It is easy to see that $\hh(T,\mu)$ is a Hilbert space.

While $\hh(T,\mu)$ consists of meromorphic functions (with simple poles in the set $T$), it is more convenient to work with 
its isomorphic copy consisting of entire functions. Let $A$ be an entire function which has simple zeros at the points
$t_n$ and no other zeros. Put
$$
\hh(T,A,\mu) = \bigg\{F(z) = A(z) \sum_n \frac{c_n\mu_n^{1/2}}{z-t_n}: \ (c_n) \in \ell^2 \bigg\}
$$
and  $\|F\|_{\hht}= \|(c_n)\|_{\ell^2}$. 
We will refer to spaces $\hht$ as the {\it Cauchy--de Branges spaces}  (CdB-spaces  for short).
It should be mentioned that the space $\hht$ is essentially determined by $(T, \mu)$ 
and spaces $\hh(T,A_1,\mu)$ and  $\hh(T,A_2,\mu)$ are canonically isomorphic to each other and to $\hh(T,\mu)$.

Any Cauchy--de Branges space $\hht$ is a Reproducing Kernel Hilbert space
(i.e., the evaluation functionals $w\mapsto F(w)$ are continuous on $\hht$ for any $w\in \co$).
We will denote by $K_w$ the reproducing kernel of $\hht$ at the point $w$.
It is easy to see that $\hht$ has the {\it Division Property}:
$$
F \in\hht, \ F(w) = 0 \ \Longrightarrow \ \frac{F(z)}{z-w} \in\hht.
$$

It also follows from the definition of the inner product in $\hht$ that the functions $F_n(z) =\mu_n^{1/2}
\frac{A(z)}{z-t_n}$ form an orthonormal basis in $\hht$. Also note that 
for $F(z) = A(z) \sum_n \frac{c_n\mu_n^{1/2}}{z-t_n}$ one has 
$$
(F, F_n)_{\hht} = c_n = \frac{F(t_n)}{\mu_n^{1/2}A'(t_n)}.
$$
Hence, $K_{t_n} = \mu_n^{1/2}\overline{A'(t_n)} F_n$ and, thus,  the system of reproducing kernels $\{K_{t_n}\}_{t_n\in T}$ 
is an orthogonal basis in $\hht$.

In fact, this property characterizes CdB-spaces.

\begin{proposition} \cite{bms2}
\label{ax1}
Let $\hh$ be a Reproducing Kernel Hilbert space which consists of entire functions and has the Division Property. 
If $\hh$ has an orthogonal basis of reproducing kernels, then there exist
$T$, $A$ and $\mu$ as above such that $\hh=\hht$ with equality of norms.
\end{proposition}

Cauchy--de Branges spaces generalize classical de Branges spaces. Namely, any de Branges space isometrically coincides
with a space $\hht$ where $T\subset \R$ and the function $A$ is real on $\R$. In particular, if $T=\Z$, $\mu_n\equiv 1$
and $A(z) = \pi^{-1}\sin \pi z$, then $\hht$ coincides with the Paley--Wiener space $PW_\pi$, 
since in this case the formula 
$$
F(z) = A(z) \sum_n \frac{c_n\mu_n^{1/2}}{z-t_n} = \sum_{n\in\Z} (-1)^n c_n\frac{\sin \pi(z-n)}{\pi(z-n)},
\qquad c_n = (-1)^n F(n),
$$
is simply the Shannon--Kotelnikov--Whittaker sampling formula. We discuss the relation between 
de Branges space and CdB-spaces in more detail in Section \ref{prelim}. 

De Branges spaces were introduced by L. de Branges in the beginning of 1960-s in his famous 
solution of the direct and inverse spectral problems for two-dimensional canonical systems
(see \cite{br} or \cite{rom}). They also turned out to be a very 
interesting object from the point of view of function theory, while operators on de Branges spaces
serve as models for various classes of abstract linear operators \cite{gt, martin, silva, by16}.
While at present there is no theory relating CdB spaces to spectral theory of a class of differential operators, 
these spaces have rich connections with many areas of function and operator theory. 
Therefore, it seems to be a noteworthy goal to extend some of the basic results of the de Branges theory 
to the more general and complicated field of CdB spaces.

A specific aim of the present (partially survey) paper is to study analogs of two important properties 
of de Branges spaces in CdB-space setting. The first part deals with the properties of the systems of reproducing kernels
corresponding to the zeros of certain entire functions associated to the space. In the case of de Branges spaces they correspond to
orthogonal bases of reproducing kernels. The second theme of the paper is a characterization of the density of 
the domain of multiplication by $z$ in $\hht$ in terms of the spectral data  $(T, \mu)$.


\subsection{De Branges orthogonal bases of reproducing kernels and their generalizations.} Given $T$, $A$ and $\mu$, 
we define for any $\gamma \in \co$ the entire function
\begin{equation}
\label{gh}
B_\gamma(z)  = A(z) \bigg(\gamma + \sum_n \Big(\frac{1}{t_n-z} -\frac{1}{t_n}\Big) \mu_n \bigg).
\end{equation}
We assume that the term in the brackets is simply $-1/z$ in case 
when $t_n=0$ (also we do not lose much in generality if we assume that $0\notin T$).

Functions $B_\gamma$ with real $\gamma$ play an important role 
in the de Branges theory. In particular, they give rise to orthogonal bases of reproducing kernels.

\begin{theorem} \cite[Theorem 22]{br}
Let $\hht$ be a de Branges space. Then for any $\gamma \in \rl$
all zeros of $B_\gamma$ are real and simple, 
for different $\gamma$-s the zero sets
$\mathcal{Z}(B_\gamma)$ of $B_\gamma$ interlace, and the family of 
reproducing kernels $\{K_w\}_{w\in \mathcal{Z}(B_\gamma)}$
is an orthogonal basis in $\hht$ for all $\gamma \in \rl$ except, maybe, one. 
\end{theorem}

Thus, in de Branges spaces there is a continuous family of orthogonal bases of reproducing kernels. This
is a special case of a more general construction of the so-called Clark measures developed later (and independently)
by D.N. Clark \cite{clark} (see, also, a survey \cite{saks}). Moreover, 
the property of having at least two orthogonal bases of reproducing kernels
distinguishes de Branges spaces among all Hilbert spaces of entire functions:

\begin{theorem} \cite{bms1}
\label{jam}
If a Reproducing Kernel Hilbert space $\hh$ of entire functions
with Division Property has two orthogonal bases of reproducing kernels, then $\hh = \hht$
where $T$ lies on a straight line. Thus, $\hh$ is a de Branges space up to a rotation.
\end{theorem}

In fact in \cite{bms1} a stronger result is proved: if a discrete weighted Hilbert transform 
is unitary, then the corresponding points lie on a line or on a circle. 
 
If $\gamma\notin \rl$, then in the de Brangean setting ($T\subset \rl$, $A$ is real on $\rl$)
the system $\{K_w\}_{w\in \mathcal{Z}(B_\gamma)}$ is no longer orthogonal, but it still can
form an unconditional basis (i.e., Riesz basis up to normalization) in $\hht$. It follows  
from the classical Hardy space theory that this is the case if and only if $\mathcal{Z}(B_\gamma)$ 
is a Carleson interpolating sequence in $\cp$ or in $\cm$ (see Proposition \ref{carl} for details). 

Zeros of functions $B_\gamma$ have one more important property: any two zero sets $\mathcal{Z}(B_\gamma)$ 
define the de Branges space up to a nonvanishing factor. 

\begin{theorem} \cite[Theorem 24]{br}
\label{twos}
Let  $\hh(T,A,\mu)$ and $\hh(\tilde T, \tilde A, \tilde \mu)$ be two de Branges spaces.
If there exist $\alpha, \beta\in \rl\setminus \{0\}$, 
$\alpha\ne \beta$, such that $\mathcal{Z}(B_\alpha) = \mathcal{Z}(\tilde B_\alpha)$ 
and $\mathcal{Z}(B_\beta) = \mathcal{Z}(\tilde B_\beta)$, then $\tilde T= T$, $\tilde \mu = \mu$
and $\tilde A = SA$ for some nonvanishing entire $S$. 
\end{theorem}

On the canonical systems side zeros of $B_\gamma$ correspond to spectra of selfadjoint extensions
of the ``canonical system operator'' (see \cite{rom}). Thus, Theorem \ref{twos}
is a de Branges space counterpart of the classical ``two spectra theorem'' of G. Borg which allows to recover the potential
of the Sturm--Liouville equation from the spectra of its Dirichlet and Neumann problems (see, e.g., \cite[Section 7]{rom}).

We are interested in the properties of the system of reproducing kernels $\{K_w\}_{w\in \mathcal{Z}(B_\gamma)}$
in the case of general CdB-spaces. We know that such system cannot be an orthogonal basis unless $T$ lies on a line. 
Several natural questions are:
\medskip
\\
{\bf Problems.} 1. Is it true that there always exist 
$\gamma$ such that $\{K_w\}_{w\in \mathcal{Z}(B_\gamma)}$ is a Riesz basis for $\hht$?
\smallskip

2. More generally, how the set of parameters
$\gamma$ such that $\{K_w\}_{w\in \mathcal{Z}(B_\gamma)}$ is a Riesz basis depends on $T$ and $\mu$? 
\smallskip

3. Is it true that the family $\{K_w\}_{w\in \mathcal{Z}(B_\gamma)}$ is at least complete in $\hht$, i.e., 
$\mathcal{Z}(B_\gamma)$ is always a uniqueness set for $\hht$ when $\gamma\ne 0$? In the case when 
$B_\gamma$ has multiple zeros this means that there is no nonzero function in $\hht$ with zeros at 
$\mathcal{Z}(B_\gamma)$ of the corresponding multiplicity.
\smallskip

4. Do the zero sets $\mathcal{Z}(B_\alpha)$ and $\mathcal{Z}(B_\beta)$ with $\alpha\ne\beta$ define $\hht$
up to a nonvanishing factor?
\medskip

These problems seem to be in general difficult. In particular,
it is not even clear whether the zero set of  $B_\gamma$ is always infinite. This is related to a problem
posed by J.~Clunie, A.~Eremenko and J.~Rossi \cite{cer}:
\medskip
\\
{\bf Conjecture} (Clunie, Eremenko, Rossi, 1993). {\it Let $a_n>0$, $t_n\in \co$, $|t_n| \to \infty$, and
$\sum_n \frac{a_n}{|t_n|+1} <\infty$. Then the function
$$
f(z) = \sum_n \frac{a_n}{z-t_n} 
$$
has infinitely many zeros. }
\medskip

The conjecture was confirmed for many special cases \cite{cer, el, lr}, however, to the best of our knowlegde, it is still open in general. 
The question for the regularized Cauchy transforms (as in the definition of $B_\gamma$) is apparently 
more complicated, since now we nead to deal with a Cauchy transform with coefficients $\mu_n/t_n$
with $\mu_n>0$. Therefore, in what follows we will distinguish a class of CdB-spaces with additional condition that
$$
 \sum_n \frac{\mu_n}{|t_n|+1} <\infty.
$$
In this case we will say that CdB-space $\hht$ belongs to the  {\it convergence class}. 
We say that $\hht$ is {\it small} if, moreover, $\sum_n \mu_n <\infty$.
For CdB-spaces of the convergence class we modify the definition of the functions $B_\gamma$ 
and put
\begin{equation}
\label{ghh}
B_\gamma(z) = A(z) \bigg(\gamma + \sum_n \frac{\mu_n}{t_n-z}\bigg).
\end{equation}

We also say that CdB-space $\hht$ is a {\it space of finite order} if all functions in $\hht$ are of finite order. 
It is not difficult to show that this is the case if and only if $A$ is of finite order (and its order majorizes the order of all elements
in $\hht$), see \cite[Lemma 2.5]{abb19}.

Under these additional restrictions on $\hht$ we are able to obtain positive answers to the above questions. 

\begin{theorem} 
\label{comp}
1. Let $\hht$ be a CdB-space of the convergence class and of finite order. Then for any $\gamma\ne 0$  
the sequence $\mathcal{Z}(B_\gamma)$ is a uniqueness set for $\hht$ \textup(counting multiplicities\textup).

2. If, moreover, the space $\hht$ is small, then $\mathcal{Z}(B_0)$ is not a uniqueness set,
but the subspace of functions vanishing on $\mathcal{Z}(B_0)$  is one-dimensional.
\end{theorem} 

It is well known (see, e.g., \cite[Theorem 6.2]{go})
that in the conditions of Theorem \ref{comp} the function $B_\gamma$ has infinitely many zeros,
and moreover, its zero set has zero defect and maximal possible order. Theorem \ref{comp} shows that
this set is also maximal in the sense of being a uniqueness set for the corresponding CdB-space.

We also can prove a variant of two spectra theorem.

\begin{theorem} 
\label{twos1}
Let 
$\hh(T,A,\mu)$ and $\hh(\tilde T, \tilde A, \tilde \mu)$ be two converegence class CdB-spaces of finite order
and let
$B_\gamma$ and $\tilde B_\gamma$ be the corresponding functions \eqref{ghh}. If there exist $\alpha, \beta\in \co\setminus \{0\}$, 
$\alpha\ne \beta$, such that $\mathcal{Z}(B_\alpha) = \mathcal{Z}(B_\alpha)$ 
and $\mathcal{Z}(B_\beta) = \mathcal{Z}(\tilde B_\beta)$ \textup(counting multiplicities\textup), then $\tilde T= T$, $\tilde \mu = \mu$
and $\tilde A = SA$ for some nonvanishing entire $S$.
\end{theorem}

We pass to the question about Riesz bases of the form $\{\tilde K_w\}_{w\in \mathcal{Z}(B_\gamma)}$,
where $\tilde K_w =  K_w/\|K_w\|$ are the normalized reproducing kernels. 
We can prove this property only in a special case
of small CdB spaces with a certain separation of $T$.

We say that the sequence $T=\{t_n\}$ is {\it power separated} (with exponent $N$) if 
there exist numbers $C>0$ and $N>-1$ such that, for any $n$,
\begin{equation}
\label{powsep}
{\rm dist}\,(t_n, T\setminus\{t_n\}) \geq C(|t_n|+1)^{-N}.
\end{equation}

\begin{theorem} 
\label{bas}
Let $\hht$ be a small \textup(i.e., $\sum_n \mu_n <\infty$\textup) CdB-space. Assume that $T$ is power separated
with exponent $N$ and $(t_n^{2N} \mu_n) \in\ell^p$ for some $p>0$. If $\gamma\ne 0$ and
the zeros of $B_\gamma$ are simple, then the family $\{\tilde K_w\}_{w\in \mathcal{Z}(B_\gamma)}$ is a Riesz basis in $\hht$.

In particular, the theorem is true if $\hht$ is a small CdB-space and $T$ is separated, i.e., 
$|t_n-t_m|\ge \delta$ for some $\delta>0$ and any $n\ne m$.
\end{theorem}

In Section \ref{riss} we prove a more general result (Theorem \ref{bas2}) about Riesz bases 
of reproducing kernels corresponding to the
zeros of some entire function which is a small, in a sense, perturbation of $A$. In Section \ref{appli}
we use the functional model for rank one perturbations of normal operators to show that 
the eigenvectors of certain rank one perturbations form a Riesz basis. Recently, O. Dobosevych 
and R. Hryniv \cite{hry1, hry2} studied possible spectra of rank one perturbations of
unbounded selfadjoint operators whose spectrum is discrete and separated. Assume that $T=\{t_n\}_{n\ge 1} \subset \rl$  
is separated and let $\A$ be an unbounded selfadjoint operator with simple spectrum $T$ on some Hilbert
space $H$. Then a set $S = \{s_n\}_{n\ge 1}$ is the spectrum of some rank one pertubation $\LL = \A+ a\otimes b$, $a,b\in H$,
if and only if $S$ can be enumerated so that $\sum_n|s_n-t_n| <\infty$ \cite[Theorems 3.1, 4.1]{hry2}. 
It is easy to extend this statement to normal operators $\A$ and also to show that in this case 
the (generalized) eigenvectors of $\LL$ form a Riesz basis in $H$ (see Theorem \ref{bas3} below). 


\subsection{Operator of multiplication by $z$ in CdB-spaces.}
In the de Branges theory an important role is played by the operator of multiplication by $z$. Clearly, this is an unbounded operator.
Given a CdB space $\hht$, the domain of the operator $M_{z^N}$ of multiplication by $z^N$, $N\in \N$, is given by 
$\mathcal{D}_{z^N} = \{ F\in\hht: \ z^N F\in \hht\}$, $M_{z^N} F=z^N F(z)$, $F\in \mathcal{D}_{z^N}$. 

The multiplication operator $M_z$ in a de Branges space serves as a model for a class
of symmetric linear operators with deficiency indices $(1, 1)$ (to be precise, for simple regular closed
operators), see \cite{martin, silva}. 

A necessary and sufficient condition for the domain of $M_z$ to be dense in a de Branges space
is given in \cite[Theorem 22]{br}.
It turns out that either $\mathcal{D}_z$ is dense, or its closure has codimension 1 and is itself
a de Branges space with respect to the initial norm. Moreover, if a de Branges space has a de Branges subspace of 
codimension 1, then it necessarily must be given by the closure of $\mathcal{D}_z$. 
Analogous results hold for $M_{z^N}$. We extend these results to the case of general CdB-spaces. 
At least some arguments of \cite{br} essentially use the symmetry with respect to $\R$ 
and cannot be applied in the general case. Also, we replace the notion of de Branges subspaces by a more general
notion of a {\it nearly invariant subspace}.

We say that a closed linear subspace $\hh_0 \subset \hht$ is {\it nearly invariant} if 
it has the Division Property itself, that is, $f\in \hh_0$, $f(w)=0$ implies that 
$\frac{f(z)}{z-w} \in \hh_0$.

\begin{theorem} 
\label{dom}
Given a CdB space $\hht$, the following are equivalent: 
\begin{enumerate} 
\item [(i)] $\hht$ contains a nearly-invariant subspace $\hh_0$ of codimension $N$\textup;  
\smallskip
\item [(ii)] ${\rm clos}\,\mathcal{D}_{z^N}$ is a subspace of $\hht$ of codimension $N$\textup; 
\smallskip
\item [(iii)] $\sum_n |t_n|^{2N-2} \mu_n <\infty$.  
\end{enumerate} 
In this case the nearly invariant subspace of codimension $N$ is unique
and is given by  $\hh_0  = {\rm clos}\,\mathcal{D}_{z^N}= \{B_0, \dots, B_{N-1}\}^\perp$,
where
\begin{equation}
\label{bj}
B_j(z) = A(z)\sum_n \frac{\overline{t_n}^j  \mu_n}{z-t_n}, \qquad j=0, \dots, N-1.
\end{equation}
\end{theorem}

In particular, ${\rm clos}\,\mathcal{D}_{z} \ne \hht$ if and only if
$\sum_n \mu_n <\infty$ ($\hht$ is small) and  ${\rm clos}\,\mathcal{D}_{z} = \{B_0\}^\perp$.
Note that the definition of $B_0$ from  \eqref{bj} coincides with the function $B_0$ given by \eqref{ghh}.

It follows from Theorem \ref{dom} that if $\sum_n |t_n|^{N} \mu_n <\infty$ for any $N$, i.e.,
$L^2(\mu)$ contains the set $\mathcal{P}$ of all polynomials, 
then $\hht$ contains subspaces ${\rm clos}\,\mathcal{D}_{z^N}$ of any finite codimension,
ordered by inclusion. The ordered structure of de Branges subspaces of a de Branges space 
is one of the most striking and important features of de Branges theory. Under certain conditions 
on the spectrum, ordered structure for nearly invariant subspaces of a CdB-space was proved in \cite[Theorems 1.3, 1.4]{abb19}.

Under certain conditions one can describe CdB spaces such that all nontrivial 
nearly invariant subspaces are of finite codimension (and, thus, are given by 
${\rm clos}\,\mathcal{D}_{z^N}$, $N\in\N$).

\begin{theorem}
\label{stro1}
Let $\hht$ be a CdB-space such that $T$ is power separated \textup(i.e., satisfies \eqref{powsep}\textup).
Then the following assertions are equivalent\textup:
\smallskip

1. $\htt$ contains a nearly invariant subspace of any finite codimension and any nontrivial 
nearly invariant subspace is of this form. 
\smallskip

2. The set of all polynomials $\pp$ is contained in $L^2(\mu)$ and is dense there.
\end{theorem}

It is a natural question, whether the subspace $\hh_0$ in Theorem \ref{dom} is itself a CdB-space with respect to
the norm inherited from $\hht$, i.e., whether it has an {\it orthogonal} basis of reproducing kernels. 
In Proposition \ref{nobas} we show that while $\hh_0$ coincides with a CdB-space with equivalence of norms, it is 
a CdB-space itself if and only if $\hht$ is a rotation of a de Branges space (meaning that all $t_n$ lie on some straight line).\
\bigskip
\\
\textbf{Organization of the paper.} In Section \ref{prelim} we discuss some basic facts concerning
de Branges and Cauchy--de Branges spaces. Theorems \ref{comp} 
and \ref{twos1} are proved in Section \ref{33}. Theorem \ref{bas} as well as a more general 
sufficient condition for being a Riesz basis of reproducing kernels are proved in Section~\ref{riss},
while in Section \ref{appli} these results are applied to the spectral theory of rank one perturbations 
of normal operators via a functional model. Two specific examples are considered in Section \ref{examp}.
Finally, the proofs of Theorems \ref{dom} and \ref{stro1} are given in Sections \ref{riv} and \ref{stru}
respectively. 
\bigskip
\\
\textbf{Notations.} In what follows we write $U(x)\lesssim V(x)$ if 
there is a constant $C$ such that $U(x)\leq CV(x)$ holds for all $x$ 
in the set in question. We write $U(x)\asymp V(x)$ if both $U(x)\lesssim V(x)$ and
$V(x)\lesssim U(x)$. The standard Landau notations
$O$ and $o$ also will be used.
The zero set of an entire function $f$ will be denoted  by $\mathcal{Z}(f)$. 
We denote by $D(z,R)$ the disc with center $z$ of radius $R$ and by 
$m_2$  the Lebesgue area measure in $\CC$.
\bigskip
\\
\textbf{Acknowledgement.} The author is grateful to Artur Nicolau for useful discussions concerning Frostman shifts, 
to Vladimir Shemyakov for the help with numerical experiments and to Antonio Rivera for
the discussions of the material of Section \ref{riv}.
\bigskip


\section{Preliminaries}
\label{prelim}

\subsection{Reproducing kernels of CdB spaces}

Let $f(z) = A(z) \sum_n \frac{c_n\mu_n^{1/2}}{z-t_n} \in \hht$. Then 
$$
f(t_n) = A'(t_n) c_n \mu_n^{1/2} =\Big( f,  \frac{A(z) \overline{A'(t_n)} \mu_n}{z-t_n} \Big).
$$ 
Thus, the functions $\frac{A(z) \overline{A'(t_n)} \mu_n}{z-t_n}$, the reproducing kernels at the points $t_n$,  
form an orthogonal basis in $\hht$. Note also that 
\begin{equation}
\label{skal}
(f, g)_{\mathcal{H}(T,A,\mu)} = \big((c_n), (d_n)\big)_{\ell^2}= 
\sum_n \frac{f(t_n) \overline{g(t_n)}}{|A'(t_n)|^2 \mu_n}
\end{equation} 
for any $g(z) = A(z) \sum_n \frac{d_n\mu_n^{1/2}}{z-t_n} \in \hht$,
and so the space $\mathcal{H}(T,A,\mu)$ is isometrically embedded into the space $L^2(\nu)$, where 
$\nu = \sum_n |A'(t_n)|^{-2} \mu_n^{-1} \delta_{t_n}$. 

If $w\notin T$, then the reproducing kernel at the point $w$ is given by 
$$
K_w(z) = A(z) \sum_n \frac{\mu_n}{(\bar w -\bar t_n)(z-t_n)}.
$$

Any CdB space has the Division Property. Indeed, if $f(w) = 0$ and $w\notin T$, then
$$
\frac{f(z)}{z-w} = A(z) \sum_n \frac{c_n \mu_n^{1/2}}{(t_n-w)(z-t_n)}.
$$
For $w=t_m\in T$, we have $f(t_m) =0$ if and ony if $c_m =0$. Then
$$
\begin{aligned}
\frac{f(z)}{z-t_m} & = A(z) \sum_{n\ne m}\frac{c_n \mu_n^{1/2}}{(z-t_m)(z-t_n)} \\
& = 
A(z) \sum_{n\ne m}\frac{c_n \mu_n^{1/2}}{(t_n-t_m)(z-t_n)} +
\frac{A(z)}{z-t_m} \sum_{n\ne m}\frac{c_n \mu_n^{1/2}}{t_m-t_n}.
\end{aligned}
$$

As mentioned in the Introduction, existence of an orthogonal basis of reproducing kernels
and the Division Property distinguishes CdB spaces among all 
Reproducing Kernel Hilbert spaces of entire functions. For the sake of completeness we outline the proof. 

\begin{proof}[Proof of Proposition \ref{ax1}] 
Assume that $\{K_{t_n}\}$ is Riesz basis of reproducing kernels in $\hh$. 
Then $K_{t_m}(t_n) = (K_{t_m}, K_{t_n})_{\hh} =0$, $n\ne m$. Fix some
$t_m$ and put $A = (z-t_m)K_{t_m}$. Then $\frac{A(z)}{z-t_n}$ belongs to $\hh$ and vanishes on $\{t_l\}_{l\ne n}$, whence 
$\frac{A(z)}{z-t_n} = a_n K_{t_n}$ for some constant $a_n$. Put $\mu_n = \big\| \frac{A(z)}{z-t_n} \big\|_{\hh}^{-2}$.
Then any element of $\hh$ can be written as the sum of an orthogonal series
$$
f(z) = \sum_n c_n \mu_n^{1/2} \frac{A(z)}{z-t_n}
$$
and $\|f\|_{\hh} = \|(c_n)\|_{\ell^2}$. Thus, $\hh = \hht$.
\end{proof}

In what follows for $f\in\hht$ it will be sometimes convenient to write $f\longleftrightarrow (c_n)$ 
in place of $f(z) = A(z) \sum_n \frac{c_n \mu_n^{1/2}}{z-t_n}$. Note that if $f(w) = 0$, $w\notin T$, one has
$\frac{f(z)}{z-w} \longleftrightarrow \big( \frac{c_n}{t_n-w}\big)$. Also, if $f\in {\mathcal D}_{z^N}$,
then $z^N f \longleftrightarrow (t_n^N c_n)$.


\subsection{De Branges spaces}

There are several equivalent ways to define de Branges spaces (see \cite{br}). 
There is an axiomatic definition: a
Reproducing Kernel Hilbert space $\hh$ which consists of entire functions is a {\it de Branges space} if
the map $F\mapsto F^*$, where $F^*(z) = \overline{F(\overline z)}$ 
is an isometry on $\hh$ and for any $F\in \hh$ and $w\in \co$ such that $F(w) =0$ 
the function $\frac{z-\bar w}{z-w} F$ belongs to $\hh$ and has the same norm as $F$. 
It is clear that if $T\subset \R$ and $A$ is real on $\R$ (i.e., $A^*=A$),
then $\hht$ is a de Branges space.

Another approach uses the notion of an Hermite--Biehler function. 
An entire function $E$ is said to be in the Hermite--Biehler class if
$E$ has no zeros in $\mathbb{C}_+ \cup\rl$ and 
$$
|E(z)| > |E^*(z)|,  \qquad z\in {\mathbb{C}_+}.
$$ 
With any such function we associate the space
$\mathcal{H} (E) $ which consists of all entire functions
$F$ such that $F/E$ and $F^*/E$ restricted to $\mathbb{C_+}$ belong
to the Hardy space $H^2=H^2(\mathbb{C_+})$.
The inner product in $\he$ is given by
$$
( F,G)_{\he}= \int_\rl \frac{F(t)\overline{G(t)}}{|E(t)|^2} \,dt.
$$
Any space $\he$ satisfies the axioms of a de Branges space and any de Branges space
(without common real zeros)
is of this form for some $E$ \cite[Theorem 23]{br}.

A function $E$ is in the Hermite--Biehler class if and only if 
$\Theta = \Theta_E = E^*/E$ is inner in $\cp$: the mapping $F\mapsto F/E$
is a unitary operator from $\mathcal{H}(E)$
onto the subspace $K_\Theta = H^2\ominus\Theta H^2$ of the Hardy space 
$H^2$ known as a {\it model subspace}.

The reproducing kernel of ${\mathcal H} (E)$
corresponding to the point $w\in \mathbb{C}$ is given by
\begin{equation}
\label{repr}
K_w(z)=\frac{\overline{E(w)} E(z) - \overline{E^*(w)} E^*(z)}
{2\pi i(\overline w-z)} =
\frac{\overline{A(w)} B(z) -\overline{B(w)}A(z)}{\pi(z-\overline w)},
\end{equation}
where entire functions $A$ and $B$ are defined by $A = \frac{E+E^*}{2}$,
$B=\frac{E^*-E}{2i}$, so that $A$ and $B$ are real on $\mathbb{R}$
and $E=A - iB$. Note that $\frac{B}{A} = i\frac{1-\Theta}{1+\Theta}$ has positive
imaginary function in $\cp$ and is real on $\R$; thus, it can be written as
\begin{equation}
\label{herg}
\frac{B(z)}{A(z)} = pz+q + \sum_n \Big(\frac{1}{t_n-z} -\frac{1}{t_n}\Big) \mu_n
\end{equation}
for some $p\ge 0$, $q\in \rl$ and $\mu_n>0$ such that $\sum_n \frac{\mu_n}{t_n^2+1} <\infty $. 
We always assume that the term in the brackets is simply $-1/z$ in case 
when $t_n=0$. The reproducing kernels $\{K_{t_n}\}_{t_n\in \ZZ(A)}$ 
form an orthogonal basis in $\he$ if and only if $p=0$. More generally, note that $\hh(e^{i\alpha} E) = \he$
for any $\alpha\in \R$, and the reproducing kernels corresponding to the zeros 
of the function $A \cos \alpha + B \sin \alpha$ (which is ``$A$-function'' for $e^{i\alpha} E$)
form an orthogonal basis in $\he$ for all $\alpha\in [0, \pi)$
except at most one \cite[Theorem 22]{br}. 

If $p=0$ (equivalently, $A\notin \he$), then 
$\{K_{t_n}\}_{t_n\in \ZZ(A)}$ form an orthogonal basis in $\he$ and it is easy to see 
that $\he = \hht$.
\bigskip


\section{Proofs of Theorems \ref{comp} and \ref{twos1}}
\label{33}

In this section one of our key tools is the following variant of Liouville's theorem. We say that $\Omega\subset \CC$ is a {\it 
set of zero area density} if 
$$
\lim_{R\to\infty} \frac{m_2(\Omega \cap D(0, R))}{R^2} = 0.
$$

The following result was proved in \cite{bbb-fock}. Its proof was based on deep estimates of the harmonic 
measure due to A. Beurling and L. Ahlfors. A nice elementary argument was later suggested by B.N. Khabibullin \cite{khab}. 
We present the proof by Khabibullin below. 

\begin{theorem} 
\label{dens}
If an entire function $f$ of finite order is bounded on 
$\CC\setminus \Omega$ for some set $\Omega$ of zero area density, 
then $f$ is a constant. 
\end{theorem}

\begin{proof}
Let $|f(z)| \le 1$ for $z\in \co\setminus  \Omega$ where  $\Omega$ has zero area density. Put $u = \max (\log |f|, 0)$.
Then $u$ is subharmonic and $u=0$ on $\co\setminus  \Omega$. For $z\in\co$, $r>0$, let
$$
M(z, r) =\frac{1}{\pi r^2} \int_{D(z,r)} u(\zeta)\, dm_2(\zeta), \qquad M(r) = M(0,r). 
$$
By subharmonicity, $u(z) \le M(z,r)$ for any $r$. Also, $M(z,r) \le 4 M(2r)$ whenever $|z| \le r$.
Then we have
$$
\begin{aligned}
M(r) = \frac{1}{\pi r^2} \int_{D(0,r) \cap\Omega} u(z)\, dm_2(z)  & \le 
 \frac{1}{\pi r^2} \int_{D(0,r) \cap\Omega} M(z,r) \, dm_2(z) \\
& \le 4 M(2r)\frac{m_2(D(0,r) \cap\Omega)}{\pi r^2}.
\end{aligned}
$$
Since $\Omega$ is of zero density, we conclude that $M(r) = o(M(2r))$, $r\to\infty$.

It is easy to see that the latter condition contradicts the fact that $f$ is of finite order unless $M(r) \equiv 0$. 
Assume that $M(r)>0$ for sufficiently large $r$.
For any $\gamma>0$ 
choose $r_0$ such that $M(r_0)>0$ and $M(2r) \ge \gamma M(r)$, $r\ge r_0$. Since 
$f$ is of finite order, $u(z) =O(|z|^\rho)$, $|z|\to\infty$, whence
$M(r) \le C r^\rho$ for sufficiently large $r$. Thus, for any $n\in\mathbb{N}$,
$M(2^n r_0) \le C 2^{\rho n} r_0^\rho$. At the same time $M(2^n r_0) \ge \gamma^n M(r_0)$. 
Since $\gamma$ can be taken arbitrarily large (e.g., $\gamma>2^\rho$) we come to a contradiction. Thus, $M(r) \equiv 0$,
whence $|f(z)| \le 1$ in $\co$.
\end{proof}

The following result contains Statement 1 of Theorem \ref{comp} 
as a special case (where $a_n=-\mu_n$).

\begin{theorem} 
\label{comp1}
Let $\hht$ be a CdB-space of finite order. Let 
$$
G(z) = A(z) \bigg(\gamma + \sum_n \frac{a_n}{z-t_n}\bigg),
$$
where $\gamma\in \co\setminus\{0\}$  and $\sum_n \frac{|a_n|}{|t_n|+1} <\infty$. 
Then the zero set $\mathcal{Z}(G)$ \textup(counting multiplicities\textup) is a uniqueness set for $\hht$.
\end{theorem}

\begin{proof}
A classical result of meromorphic function theory (see  \cite[Theorem 6.1]{go}) says that if
$\sum_n \frac{|a_n|}{|t_n|+1} <\infty$, then $f(z) = \sum_n \frac{a_n}{z-t_n}$ satisfies
$$
\int_0^{2\pi} |f(re^{i\phi})|^p d\phi \to 0, \qquad r\to \infty,
$$
for any $p\in (0,1)$. From this we easily obtain that
$|f(z)| \le |\gamma|/2$ on $\co\setminus\Omega$,
where $\Omega$ has zero area density, whence $|G| \ge |\gamma|\cdot |A|/2$ on $\co\setminus\Omega$.

Assume that there exists $F(z)  = A(z) \sum_n \frac{c_n\mu_n^{1/2}}{z-t_n} \in\hht$ which vanish on $\mathcal{Z}(G)$
counting multiplicities. Then we can write $F=G H$ for some entire function $H$ and so
$$
H = \frac{F}{G} = \frac{A}{G} f_1, \qquad f_1(z) = \sum_n \frac{c_n\mu_n^{1/2}}{z-t_n}.
$$
Applying \cite[Theorem 6.1]{go} to $f_1(z)$ we conclude that $|f_1(z)| \le 1$ outside another set $\Omega_1$ of zero area density. 
We conclude that $H$ is bounded in $\co\setminus(\Omega \cup \Omega_1)$. At the same time $H$ is of finite order since 
$H=F/G$ and both $F$ and $G$ are of finite order. By Theorem \ref{dens} $H$ is constant and so $G\in\hht$. 

It remains to note that $G\notin \hht$ for any $\gamma\ne 0$. Indeed, if $G \in \hht$, then, 
comparing the values at $t_n$ we conclude that $G (t_n) = A'(t_n) a_n$, whence $(\mu_n^{-1/2}a_n) \in \ell^2$ by \eqref{skal}. 
It follows that $A(z) \sum_n \frac{a_n}{z-t_n} \in \hht$ and, finally, $\gamma A\in \hht$, which is absurd. 
\end{proof}
\medskip

\begin{proof}[Proof of Theorem \ref{comp}]
Statement 1 follows from Theorem \ref{comp1} applied to $a_n= - \mu_n$. Let the space $\hht$ be small.
Then it is clear that $B_0 \in \hht$, $B_0\longleftrightarrow (\mu_n^{1/2})$. Thus, $\mathcal{Z}(B_0)$ is not a uniqueness set. 
Let $F\in \hht$ vanish on $\mathcal{Z}(B_0)$. We need to show that $F$ is a multiple of $B_0$. 

We have
$$
\frac{zB_0(z)}{A(z)} = -\sum_n \mu_n +\sum_n \frac{\mu_n t_n}{t_n -z}.
$$
Then, by \cite[Theorem 6.1]{go} (as in the proof of Theorem \ref{comp1}) we conclude that $|B_0(z)| \gtrsim |z|^{-1}|A(z)|$
outside a set of zero area density. If $F(z) = \sum_n \frac{c_n\mu_n^{1/2}}{z-t_n}$ vanish on $\mathcal{Z}(B_0)$,
we can write $F=B_0H$ for some entire function $H$ of finite order, whence
$$
\frac{B_0(z) H(z)}{A(z)} = \sum_n \frac{c_n\mu_n^{1/2}}{z-t_n} = o(1)
$$
as $|z|\to \infty$ outside another set of zero area density. We conclude that $|H(z)| = o(|z|)$  outside a set of zero density
and so $H=const$ by Theorem \ref{dens}.
\end{proof}
\medskip

\begin{proof}[Proof of Theorem \ref{twos}]
Assume that
$\mathcal{Z}(B_\alpha) = \mathcal{Z}(\tilde B_\alpha)$ 
and $\mathcal{Z}(B_\beta) = \mathcal{Z}(\tilde B_\beta)$ (counting multiplicities). Then 
$\tilde B_\alpha = g_1 B_\alpha$ and $\tilde B_\beta = g_2 B_\beta$ where $g_1, g_2$ are nonvanishing entire functions. 
Dividing one of these equation by the other, we get
$$
\bigg(\alpha + \sum_n \frac{\tilde \mu_n}{z-\tilde t_n} \bigg) 
\bigg(\beta + \sum_n \frac{ \mu_n}{z- t_n} \bigg)  = g(z)
\bigg(\alpha + \sum_n \frac{\mu_n}{z- t_n} \bigg) 
\bigg(\beta + \sum_n \frac{ \tilde \mu_n}{z- \tilde t_n} \bigg),
$$
where $g = g_1/g_2$ is a nonvanishing entire function. 
Arguing as in the proof of Theorem \ref{comp1} (making use of \cite[Theorem 6.1]{go})
we conclude that each of the brackets is bounded and bounded away from zero
on $\co\setminus\Omega$, where $\Omega$ has zero area density,
whence $g$ is bounded on $\co\setminus\Omega$. Since $\hht$ is a space 
of finite order, $g$ also is of finite order, and, by Theorem \ref{dens}, $g$ is a constant. Since, moreover, 
$$
\alpha + \sum_n \frac{\mu_n}{z- t_n} \to \alpha
$$
as $z\to \infty$ outside a set of zero area density, we conclude that $g\equiv 1$. Hence,
$$
(\beta-\alpha) \bigg(\sum_n \frac{ \mu_n}{z- t_n} - \sum_n \frac{ \tilde \mu_n}{z- \tilde t_n} \bigg)=0,
$$
whence $T=\tilde T$, $\tilde \mu= \mu$ and $\tilde A =SA$ for some entire nonvanishing $S$.
\end{proof}

The two spectra theorem remains true if one takes $T$ as one of the spectra.

\begin{corollary}
\label{ed}
Let  $\hh(T,A,\mu)$ and $\hh(T, \tilde A, \tilde \mu)$ be two convergence class CdB-spaces of finite order.
If there exists $\alpha \in \co\setminus \{0\}$, such that $\mathcal{Z}(B_\alpha) = \mathcal{Z}(\tilde B_\alpha)$
\textup(counting multiplicities\textup), 
then $\tilde \mu = \mu$ and $\tilde A = SA$ for some nonvanishing entire $S$. 
\end{corollary}

\begin{proof}
Since $\mathcal{Z}(B_\alpha) = \mathcal{Z}(\tilde B_\alpha)$ counting multiplicities, we have
$$
A(z) \bigg(\alpha + \sum_n \frac{\mu_n}{z-t_n} \bigg) 
= g(z) \tilde A(z) \bigg(\alpha + \sum_n \frac{\tilde \mu_n}{z- t_n} \bigg)
$$
for some nonvanishing entire function $g$. Arguing as above we conclude that $g\tilde A/A =1$, whence $\tilde \mu =\mu$
and $\tilde A = SA$ for some nonvanishing entire $S$. 
\end{proof}
\bigskip


\section{Riesz bases of reproducing kernels} 
\label{riss}

We start with the following elementary proposition.

\begin{proposition}
\label{hry}
Given two sequence $T = \{t_n\}_{n\ge 1}$, $|t_n| \to \infty$, and $\{a_n\}_{n\ge 1}$, 
assume that there exist $\delta_n>0$ such that the discs $D(t_n, \delta_n)$ are pairwise disjoint
and 
\begin{equation}
\label{pro2}
\sum_n \frac{|a_n|}{|z-t_n|} \to 0 \qquad \text{as}\ \  |z| \to\infty, \ z\notin \cup_n D(t_n, \delta_n).
\end{equation}
Put
\begin{equation}
\label{pro}
G(z) = A\bigg(1+ \sum_{n\ge 1} \frac{a_n}{z-t_n}\bigg). 
\end{equation}
Then there exists an enumeration of the zero set $\mathcal{Z} (G)$
\textup(counted with multiplicities\textup), $\mathcal{Z} (G) = \{s_n\}_{n\ge 1}$, such that
\begin{equation}
\label{pro1}
|s_n-t_n| \asymp |a_n|.
\end{equation}

Conversely, if there exist $T$ and  
$\delta_n>0$ such that the discs $D(t_n, \delta_n)$ are pairwise disjoint and 
the set $S= \{s_n\}_{n\ge 1} \subset \co$ satisfies 
\begin{equation}
\label{pro2a}
\sum_n \frac{|s_n-t_n|}{|z-t_n|} \to 0 \qquad \text{as}\ \  |z| \to\infty, \ z\notin \cup_n D(t_n, \delta_n), 
\end{equation}
then $S=\mathcal{Z} (G)$ 
for a function $G$ of the form \eqref{pro} with $|a_n| \asymp |s_n-t_n|$.
\end{proposition}

\begin{proof}
Clearly, $G(t_n) = 0$ when $a_n =0$.
By Rouch\'e theorem, there exists $n_0$ such that for any $n> n_0$ the functions 
$\gamma (z-t_n)+a_n$ and $(z-t_n)\big(\gamma+ \sum_{k\ge 1} \frac{a_k}{z-t_k}\big)$
have the same number of zeros in $D(t_n, \delta_n)$
whence $G$ has a unique zero $s_n \in D(t_n, \delta_n)$ and $|s_n - t_n| \asymp |a_n|$. 

Now we can write
$$
G(z) = A(z) \prod_{n>n_0} \frac{z-s_n}{z-t_n} \cdot \frac{H(z)}{\prod_{n\le n_0}(z-t_n)}
$$
for some entire function $H$. The infinite product converges uniformly
on compact subsets of the plain, since 
$$
\sum_n \frac{|s_n-t_n|}{|z-t_n|} \to 0
$$
uniformly outside of small neighborhoods of $t_n$. An application of the maximum principle completes the argument. 

It remains to show that $H$ is a polynomial of degree exactly $n_0$. Indeed, 
we have
$$
\bigg|\frac{G(z)}{A(z)}\bigg| \asymp 1, \qquad 
\prod_{n>n_0} \bigg|\frac{z-s_n}{z-t_n} \bigg|  = \prod_{n>n_0} \bigg|1 + \frac{t_n-s_n}{z-t_n} \bigg| \asymp 1,
$$
for $|z| >R$ (with a sufficiently large $R$) and $z\notin \cup_n D(t_n, \delta_n)$. We conclude that 
$H$ is a polynomial of degree $n_0$.

To prove the converse, define $G(z) = A(z)\prod_{n\ge 1} \frac{z-s_n}{z-t_n}$
and put 
$$
a_n ={\rm Res}_{t_n} \frac{G}{A} = (t_n-s_n)\prod_{k\ne n} \frac{s_n-s_k}{s_n-t_k}.
$$
By \eqref{pro2a} the product converges and $|a_n| \asymp |s_n-t_n|$. 
Here we use that $|s_n-t_n| = o(\delta_n)$, $n\to\infty$, whence $\sum_{k\ne n} \frac{|t_k-s_k|}{|s_n-t_k|} \to 0$
as $n\to\infty$. We can therefore  write
$$
\frac{G(z)}{A(z)} = \sum_n \frac{a_n}{z-t_n} +H(z)
$$
for some entire $H$. Note that $|G(z)/A(z)| \asymp 1 $ 
when $z\notin \cup_n D(t_n, \delta_n)$ and $|z|$ is sufficiently large. 
Thus, making use of  \eqref{pro2a}, we conclude that $H$ is a nonzero constant.
\end{proof} 

The following lemma provides some natural sufficient conditions for \eqref{pro2}. 
Recall that the sequence $T=\{t_n\}$ is {\it power separated} (with exponent $N$) if 
there exist numbers $C>0$ and $N>-1$ such that, for any $n$, 
${\rm dist}\,(t_n, T\setminus\{t_n\}) \geq C(|t_n|+1)^{-N}$. 
Note that we allow $N$ to be negative. If $N=0$,
then $T$ is simply separated, i.e., $|t_n-t_m|\ge \delta$ for some $\delta>0$ and any $n\ne m$.

Any power separated sequence has finite convergence exponent and so $\hht$ is a space of finite order. 
A typical (and in a sense ``the largest'') example of a power separated sequence with the exponent $N$ is 
$\{m^\alpha + i n^\alpha\}_{m,n\ge 1}$ where $\alpha = \frac{1}{1+N}$.

\begin{lemma}
\label{bbb}
Assume that $T = \{t_n\}$ is power separated with exponent $N$. If $(t_n^N a_n) \in \ell^p$
for some $p <2$, then \eqref{pro2} holds for $\delta_n = \frac{C}{3}(|t_n|+1)^{-N}$.
If, additionally, $T$ lies in a finite union of some strips, then \eqref{pro2}  holds 
when $(t_n^N a_n) \in \ell^p$ for some $p\in (0, \infty)$. 
\end{lemma}

\begin{proof}
It is sufficient to prove the lemma for the sequence
$\{m^\alpha + i n^\alpha\}_{m,n\ge 1}$ where $\alpha = \frac{1}{1+N}$.
In the general case, appending the sequence if necessary, we can always consider $T$
to be a small perturbation of a dilation of $\{\pm m^\alpha \pm i n^\alpha\}_{m,n\ge 1}$. We omit the technicalities.

For $t_n = k^\alpha+ i m^\alpha \ne k_0^\alpha+ i m_0^\alpha$ one has
$$
|k^\alpha+i m^\alpha - (k_0^\alpha+i m_0^\alpha)| \gtrsim |t_n|^{-N} (|k-k_0| + |m-m_0|).
$$
Since $\sum_{(k, m) \ne (k_0, m_0)}  (|k-k_0| + |m-m_0|)^{-q} <\infty$ if and only if $q>2$, the condition 
$(t_n^N a_n) \in \ell^p$ for some $p <2$ is sufficient for \eqref{pro2}.
\end{proof}

The following result contains Theorem \ref{bas} as a special case (where $c_n = -\gamma^{-1}\mu_n^{1/2}$,
$(\mu_n) \in \ell^1$ and $(t_n^N \mu_n) \in\ell^p$).
Recall that we denote by $\tilde K_w$ the normalized reproducing kernel of $\hht$
at a point $w$.

\begin{theorem}
\label{bas1}
Let $\hht$ be a CdB-space such that $T$ is power separated with exponent $N$ and 
$(t_n^{2N} \mu_n) \in\ell^p$ for some $p\in (0, \infty)$.
Let
$$
G(z) = A\bigg(1+ \sum_{n\ge 1} \frac{c_n \mu_n^{1/2}}{z-t_n}\bigg),
$$
where $(c_n) \in\ell^2$. If all zeros of $G$ are simple, then $\{\tilde K_s\}_{s\in \mathcal{Z}(G)}$ is a Riesz basis in $\hht$.

If $T$ lies in a finite union of some strips, then the conclusion of the theorem remains true under a weaker
condition $(t_n^{2N} \mu_n) \in\ell^\infty$. 
\end{theorem}

\begin{proof}
If $(t_n^{2N} \mu_n) \in\ell^p$, $p\in(0,1)$ and $(c_n) \in\ell^2$, we have
 $(t_n^N c_n \mu_n^{1/2}) \in \ell^s$ for some $s<2$, while 
in the case when $T$ lies in a finite union of some strips and  $(t_n^{2N} \mu_n) \in\ell^\infty$
we have $(t_n^N c_n \mu_n^{1/2}) \in \ell^2$. By Lemma \ref{bbb}, \eqref{pro2} is satisfied 
for $a_n = c_n \mu_n^{1/2}$.

Let $\mathcal{Z}(G) = \{s_n\}$. According to Proposition \ref{hry} 
we may assume that $s_n$ are enumerated so that $|s_n-t_n| \asymp |c_n| \mu_n^{1/2}$.
We will show that $\{\tilde K_{s_n}\}$ is a quadratic perturbation of the orthonormal basis $\{\tilde K_{t_n}\}$. 

Recall that, for $s_n\notin T$, one has
$$
K_{s_n} (z) =\overline{A(s_n)} A(z) \sum_k \frac{\mu_k}{(\bar s_n-\bar t_k) (z-t_k)}, \qquad
K_{t_n}  (z) = \overline{A'(t_n)} \mu_n \cdot \frac{A(z)}{z-t_n},
$$
whence $\|K_{t_n}\|^2 = |A'(t_n)|^2 \mu_n$ and
$$
\|K_{s_n}\|^2 = |A(s_n)|^2 \sum_k \frac{\mu_k}{|s_n - t_k |^2}. 
$$
Similar to the proof of Lemma \ref{bbb}, it is easily seen that
if $T$ is separated with the exponent $N$ and $(t_n^{2N} \mu_n) \in\ell^p$, then
$$
\sum_{k \ne n} \frac{\mu_k}{|s_n - t_k |^2}  =O(1).
$$
Thus, 
$$
\|K_{s_n}\|^2  \asymp \frac{|A(s_n)|^2 \mu_n}{|s_n-t_n|^2 }.
$$

We will show that  
\begin{equation}
\label{agani}
\sum_n \bigg\| \frac{(\bar s_n - \bar t_n) K_{s_n}}{\overline{A(s_n)} \mu_n^{1/2}}  -  \frac{K_{t_n}}{\overline{A'(t_n)} \mu_n^{1/2}} \bigg\|^2 <\infty.
\end{equation}
Indeed, we have
$$
\frac{(\bar s_n - \bar t_n) K_{s_n}}{\overline{A(s_n)} \mu_n^{1/2}}  -  \frac{K_{t_n}}{\overline{A'(t_n)} \mu_n^{1/2}}
= A(z) \frac{\bar s_n - \bar t_n}{\mu_n^{1/2}} \sum_{k\ne n} \frac{\mu_k}{(\bar s_n-\bar t_k) (z-t_k)}.
$$
Hence, making use of $|s_n-t_n| \asymp |c_n| \mu_n^{1/2}$,
$$
\bigg\| \frac{(\bar s_n - \bar t_n) K_{s_n}}{\overline{A(s_n)} \mu_n^{1/2}}  -  \frac{K_{t_n}}{\overline{A'(t_n)} \mu_n^{1/2}} \bigg\|^2 
\asymp |c_n|^2   \sum_{k\ne n} \frac{\mu_k}{|s_n-t_k|^2} \lesssim |c_n|^2,
$$
which proves  \eqref{agani}. 

To prove the theorem we would like to refer to the classical Bari's theorem about systems which are quadratically close to 
an orthonormal basis.  For this one needs that the system in question is quadratically independent. This is easy to show, but 
since we deal with systems of reproducing kernels, we also can use their special properties. Recall that a sequence 
of normalized reproducing kernels $\{\tilde k_\lambda\}_{\lambda\in\Lambda}$ in a Reproducing Kernel Hilbert space of analytic 
functions $\hh$ is a Riesz basis if and only if the map 
$$
f\mapsto \{f(\lambda)/\|k_\lambda\|\}_{\lambda\in\Lambda}
$$
is an isomorphism of $\hh$ onto $\ell^2$. It is well known and easy to see
that if $\hh$ has the Division Property, then the  property to be a Riesz basis
is stable with respect to moving a finite number of points. More precisely, 
for the sets $\Lambda = \{\lambda_n\}_{n\ge 1}$ and $\Tilde \Lambda = \{\tilde \lambda_n\}_{n\ge 1}$ of {\it distinct} points 
such that $\lambda_n = \tilde \lambda_n$ for $n\ge n_0$ the families 
$\{\tilde k_\lambda\}_{\lambda\in\Lambda}$ and $\{\tilde k_{\tilde \lambda}\}_{\lambda\in\tilde \Lambda}$ 
are or are not Riesz bases simultaneously. 

Now choose $n_0$ such that
$$
\sum_{n>n_0} \bigg\| \frac{(\bar s_n - \bar t_n) K_{s_n}}{\overline{A(s_n)} \mu_n^{1/2}}  -  \frac{K_{t_n}}{\overline{A'(t_n)} \mu_n^{1/2}} \bigg\|^2 <1.
$$
Then $\{\tilde K_{s_n}\}_{n>n_0}\cup \{\tilde K_{t_n}\}_{n\le n_0}$ is a Riesz basis, and so is $\{\tilde K_{s_n}\}$.
\end{proof}

\begin{remark}
{\rm 1. In the case when $G$ has a multiple zero, say, 
a zero $\lambda$ of multiplicity $m>1$, one should add to the system $\{\tilde K_s\}_{s\in \mathcal{Z}(G)}$
the (normalized) reproducing kernels for derivatives $K^{(j)}_\lambda$, $1\le j\le m-1$,
where $K^{(j)}_\lambda$ is a function in $\hht$ such that $(F, K^{(j)}_\lambda) = F^{(j)}(\lambda)$, $F\in \hht$. 
\smallskip

2. In applications to spectral theory the power separation condition is often relaxed to some weaker ``separation in the mean''
conditions (see, e.g., \cite[Sections 6, 7]{shkal}). 
In this case, using similar methods, one can show that the corresponding system of kernels will be a {\it Riesz basis with brackets}.
We will not pursue this idea here. }
\end{remark}
\bigskip


\section{Applications to rank one perturbations of normal operators}
\label{appli}

Following L. de Branges \cite{br}, we say that an entire function $G$ is {\it associated} to the space 
$\mathcal{H}(T,A,\mu)$  and write $G\in {\rm Assoc}\,(\mathcal{H}(T, A, \mu))$ 
if, for any $F\in \mathcal{H}(T,A,\mu)$ and $w\in\CC$, we have 
$$
\frac{F(w)G(z) - G(w)F(z)}{z-w} \in \mathcal{H}(T,A,\mu).
$$
If $G$ has zeros, then the inclusion $G\in {\rm Assoc}\,(\mathcal{H}(T, A, \mu))$ is equivalent to
$\frac{G(z)}{z-\lambda} \in \mathcal{H}(T,A,\mu)$ for some (any) 
$\lambda\in \mathcal{Z}(G)$. In particular, we have $A\in {\rm Assoc}\,(\mathcal{H}(T, A, \mu)) \setminus \mathcal{H}(T,A,\mu)$.

CdB-spaces form a natural setting for a functional model of rank one perturbations  
of compact normal operators. Let $\A$ be a compact normal operator with simple
point spectrum $\{s_n\}$ and trivial kernel (i.e., $0\notin \{s_n\}$). Thus,
$\A$ is unitarily equivalent to multiplication by $z$ in $L^2(\nu)$, 
$\nu = \sum_n \delta_{s_n}$.  Put $t_n = s_n^{-1}$, $T=\{t_n\}$.

For $a=(a_n)$, $b=(b_n) \in L^2(\nu)\cong \ell^2$, consider the rank one perturbation of $\A$,
$$
\LL = \A + a\otimes b, \qquad \LL x = \A x+ (x, b)a.
$$
Assume that $b=(b_n)$ is a cyclic vector for $\A$, i.e., $b_n \ne 0$ for any $n$.

The following functional model for $\LL$ was studied in \cite[Theorem 2.9]{bar18}.
It is based on the standard representation of the resolvent as a Cauchy-type integral. 

\begin{theorem}[Functional model for rank one perturbations]
\label{model}
1. Let $\A$ and $\LL$ be as above. Then
there exist 
\begin{itemize}
\item
a positive measure $\mu=\sum_n\mu_n\delta_{t_n}$ such that
$\sum_n \frac{\mu_n}{|t_n|^2 +1}<\infty$\textup; 
\item
a space $\mathcal{H}(T,A,\mu)$\textup;
\item 
an entire function $G\in  {\rm Assoc}\,(\mathcal{H}(T, A, \mu))$ with $G(0)=1$
\end{itemize}
such that $\LL$ is unitarily equivalent to the model operator 
$\mathcal{T}_G: \mathcal{H}(T,A,\mu) \to \mathcal{H}(T,A,\mu)$, 
$$
(\mathcal{T}_Gf)(z) = \frac{f(z) - f(0)G(z)}{z}, \qquad f \in \mathcal{H}(T,A,\mu).
$$

2. Conversely, for any space $\mathcal{H}(T,A,\mu)$ with $0\notin T$, and the function 
$G \in  {\rm Assoc}\,(\mathcal{H}(T, A, \mu))$ with $G(0) = 1$ 
the corresponding operator $\mathcal{T}_G$ is a model
of a rank one perturbation for some compact normal operator $\A$
with spectrum $\{s_n\}$, $s_n = t_n^{-1}$.

3. Moreover, the set of eigenvalues of $\LL$ coincides with $\{\lambda^{-1}:\ \lambda\in \ZZ(G) \}$
and the corresponding eigenvectors of the model operator $\mathcal{T}_G$ are given 
by $\frac{G(z)}{z-\lambda}$, $\lambda\in \ZZ(G)$, while 
the eigenvectors of $\mathcal{T}_G^*$ are of the form $K_\lambda$, $\lambda\in \ZZ(G)$.
\end{theorem}

The measure $\mu$ and the function $G$ in this model are related to the perturbation by the formulas
$$
\mu_n = |t_n|^2 |b_n|^2
$$
and
$$
G(z)  = A(z)\bigg( 1+ 
z \sum_n \frac{a_n \bar b_n t_n }{z- t_n} \bigg) =
A(z) \bigg( 1+ 
\sum_n a_n \bar b_n t^2_n \Big(\frac{1}{z- t_n} +\frac{1}{t_n} \Big) \bigg).
$$

Note that $\LL^*$ also is a rank one perturbation of a normal operator $\A^*$, and applying
the above model to $\LL^*$ we reduce the properties of eigenvectors of $\LL$ to the study of geometric properties
of systems of reproducing kernels in CdB spaces. In case of multiple zeros, one should also consider reproducing kernels for derivatives
$K_\lambda^{(j)}$.
Also, note that the systems $\Big\{\frac{G(z)}{G'(\lambda)(z-\lambda)}\Big\}_{\lambda\in \ZZ(G)}$
and $\{K_\lambda\}_{\lambda\in \ZZ(G)}$ are biorthogonal.

Now assume that $\A$ belongs to some Schatten class $\mathfrak{S}_p$, $p>0$, 
which is equivalent to the property that $T =\{t_n\}$ has finite convergence exponent (and so the space $\hht$ can be chosen to be of finite order).
To apply Theorem \ref{comp1} to $\hht$ and $G$ one must impose the following conditions:
$$
\sum_n |a_nb_n t_n| <\infty, \qquad  1+ \sum_n a_n \bar b_n t_n \ne 0.
$$
Under these conditions eigenvectors and root vectors of the perturbed operator $\LL$ are complete, 
the same holds for its adjoint $\LL^*$ (\cite[Theorem 2.1]{bar18}). 

In the case when $(t_n a_n) \in L^2(\nu)$ or $(t_n b_n) \in L^2(\nu)$, we are in the situation
of the so-called {\it weak perturbations} in the sense of V.I.~Macaev. Indeed, in this case $a\in {\rm Range}\,\A$ or
$b\in {\rm Range}\,\A$, whence $\LL$ can be written as $\LL = \A (I+S)$ or
$\LL = (I+S)\A $ for some rank one operator $S$. Perturbations of this form were considered in now classical
theorems of M.V.~Keldysh and V.I.~Macaev. In particular, a theorem due to Keldysh 
says that if $\A\in \mathfrak{S}_p$, $p>0$, is a normal operator, 
whose spectrum lies on a finite system of rays and $S$ is an arbitrary compact operator with ${\rm Ker}\, (I+S) =0$,
then eigenvectors and root vectors of $\LL$ are complete. In Macaev's theorem $\A$ is assumed to be only compact, 
but then it must be selfadjoint. For these results we refer to the original papers by Keldysh and Macaev \cite{keld1, keld2, Mats61}, 
as well as to \cite[Chapter V]{gk} and to a recent survey paper \cite[Section 4]{shkal}.

Our situation is much more special since $S$ of rank one, however we do not need any requirements 
on the location of the spectrum of $\A$. Note that in our case condition ${\rm Ker}\, (I+S) =0$
coincides with $1+ \sum_n a_n \bar b_n t_n \ne 0$. Also it should be mentioned that in general rank one perturbations
need not be complete and, in some cases, can even be Volterra operators (see \cite[Theorems 1.1, 1.2]{by15} 
and \cite[Theorem 8.1]{bar18} for details).

Now we consider an application to the question whether eigenvectors of a perturbed operator
form a Riesz basis. To apply Theorem \ref{bas1} to $G$ we need to write $a_n \bar b_n t^2_n = c_n\mu_n^{1/2}$ 
for some $(c_n)\in\ell^2$. Obviously, this is possible if and only if $(a_nt_n) \in \ell^2$, and so we are in the situation of 
a weak perturbation. We have the following theorem:

\begin{theorem}
\label{bas2}
Let $\A$ be a normal operator in a Hilbert space $H \cong \ell^2$ with simple
spectrum and trivial kernel. Let $t_n= s_n^{-1}$ be the inverse to the eigenvectors of $\A$
and assume that $T =\{t_n\}$ is power separated with exponent $N$ \textup(whence $\A$ is in some Schatten class\textup). 

Let $\LL = \A+a\otimes b$ where $a,b \in H$, $(t_na_n) \in\ell^2$
and $(t_n^{N+1} b_n) \in \ell^p$ for some $p\in (0, \infty)$. Assume that ${\rm Ker}\, \LL =0$
and $b=(b_n)$ is a cyclic vector for $\A$, i.e., $b_n \ne 0$ for any $n$.
Then the \textup(normalized\textup) eigenvectors and root vectors of $\LL$ form a Riesz basis in $H$.

If $T$ lies in a finite union of some strips, then the conclusion of the theorem remains true under a weaker
condition $(t_n^{N+1} b_n) \in\ell^\infty$. 
\end{theorem}

\begin{proof}
Considering the functional model of Theorem \ref{model} for $\LL$ 
we find the space $\hht$, $\mu_n = |t_n|^2 |b_n|^2$, and the corresponding function $G$ given by
\begin{equation}
\label{bho}
G(z)  = A(z)\bigg( 1+ \sum_n  a_n \bar b_n t_n
+ \sum_n \frac{a_n \bar b_n t^2_n}{z-t_n}\bigg).
\end{equation}
If all zeros of $G$ are simple, then, by Theorem \ref{bas1}
(applied to $c_n = a_n t_n$ and $\mu_n = |t_n|^2 |b_n|^2$), the normalized reproducing 
kernels $\{\tilde K_\lambda\}_{\lambda\in \ZZ(G)}$ form a Riesz basis in $\hht$. Therefore, its biorthogonal system 
(which coincides with the set of eigenvectors of $\mathcal{T}_G$, unitary equivalent model for $\LL$) also is a Riesz basis. 

By Proposition \ref{hry} and Lemma \ref{bbb}, all zeros of $G$ except, maybe, a finite number 
are simple. Thus, there exists at most finite number of multiple eigenvalues, and it is easy to see that the union of
eigenvectors and of (finite number of) root vectors gives a Riesz basis.
\end{proof}

Finally, consider a related problem about rank one perturbations of {\it unbounded} normal operators
with discrete spectrum. Now let $\A$ be an unbounded normal operator on a separable Hilbert space 
$H$ with simple spectrum $T =\{t_n\}$, $|t_n| \to \infty$, $0\notin T$. We then may assume
that  $\A$ is multiplication by $z$ in $L^2(\nu) \cong \ell^2$, $\nu=\sum_n \delta_{t_n}$.

Let $\LL = \A + a \otimes b$ be a rank one perturbation of $\A$. We will assume that $a =(a_n) \in L^2(\nu)$.
However, it is sufficent to assume that $b \in \A L^2(\nu)$, i.e., $t_n^{-1} b_n \in \ell^2$.
In this case for $x\in {\mathcal D}(\A)$ we can understand the inner product $(x,b)$ as $(\A x, (\A^*)^{-1}b)$, and so $\LL$
is well defined on ${\mathcal D}(\A)$.

Assume that ${\rm Ker}\, \LL= 0$, which is, obviously equivalent to 
\begin{equation}
\label{kap}
\varkappa = 1+(\A^{-1}a, b) = 1+\sum_n \frac{a_n \bar b_n}{t_n} \ne 0.
\end{equation}
In this case $\LL^{-1}$ is a (bounded) rank one perturbation of a compact normal operator $\A^{-1}$,
$$
\LL^{-1} = \A^{-1} - \frac{1}{\varkappa}  \A^{-1} a \otimes (\A^{-1})^*b.
$$

Assume that $b_n \ne 0$ for any $n$. 
Considering the model for $\LL^{-1}$ we see that $\LL^{-1}$ is unitarily equivalent to
an operator $\mathcal{T}_G$ on $\hht$ where $\mu_n = |b_n|^2$ and 
\begin{equation}
\label{jjk}
G(z)  = \frac{1}{\varkappa} A(z) \bigg(1 - \sum_n \frac{a_n \bar b_n}{z- t_n}\bigg)
\end{equation}
(note that $\A^{-1} a  = (t_n^{-1} a_n)$, $(\A^{-1})^*b = (\overline{t_n^{-1}} b_n)$. 

\begin{theorem}
\label{bas3}
Let $\A$ be an unbounded normal operator
in a Hilbert space $H\cong\ell^2$  whose spectrum $T=\{t_n\}$ is simple and power separated with exponent $N$.
Let $\LL = \A+a\otimes b$, where $a=(a_n) \in\ell^2$, $b_n\ne 0$ for any $n$ and 
$(t_n^N b_n) \in \ell^p$ for some $p\in(0, \infty)$. 
Then the normalized eigenvectors and root vectors of $\LL$ form a Riesz basis in $H$.

If $T$ lies in a finite union of some strips, then the conclusion of the theorem remains true under a weaker
condition $(t_n^{N} b_n) \in\ell^\infty$. 
\end{theorem}

\begin{proof}
We can assume without loss of generality that ${\rm Ker}\, \LL= 0$ (one can always replace $\A$ by $\A +\alpha I$). 
Then one can apply the above functional model to 
$\LL^{-1}$ and get the corresponding space $\hht$ with $\mu_n = |b_n|^2$ and $G$
given by \eqref{jjk}. By the hypothesis, $(t_n^{2N} \mu_n) \in \ell^p$ for some $p\in(0, \infty)$
($(t_n^{2N} \mu_n) \in \ell^\infty$ in the case  when $T$ lies in a finite union of strips). 

Now, applying Theorem \ref{bas1} (with $\mu_n = |b_n|^2$ and $c_n = a_n\bar b_n/|b_n|$), we see that
the set $\{K_{\lambda}\}_{\lambda \in \mathcal{Z}(G)}$ (with added finite number of functions 
$K^{(j)}_{\lambda}$ in case of multiple zeros) is a Riesz basis in $\hht$.
Recall that in our model $\{K_{\lambda}\}_{\lambda \in \mathcal{Z}(G)}$ 
is the system of eigenvectors of  
$\mathcal{T}_G^*$ (which is unitary equivalent to $(\LL^{-1})^*$). Its biorthogonal system is also 
a Riesz basis, thus, the eigenvectors of $\LL^{-1}$ (and of $\LL$) form a Riesz basis in $H$.  
\end{proof}

\begin{corollary}
\label{bas4}
Let $\A$ be an unbounded normal operator
in a Hilbert space $H\cong\ell^2$  whose spectrum $T=\{t_n\}$ is simple and separated,
and let $\LL = \A+a\otimes b$, $a,b\in H$. Then there exists an enumeration $\{z_n\}$ 
of the spectrum of $\LL$ \textup(counting multiplicities\textup) such that $\sum_n |z_n-t_n| <\infty$. Moreover, the normalized 
eigenvectors and root vectors of $\LL$ form a Riesz basis in $H$.

Conversely, for any sequence $\{z_n\}$ such that $\sum_n |z_n-t_n| <\infty$ there exists a unique rank one perturbation 
$\LL = \A+a\otimes b$, $a,b\in H$, of $\A$ such that the spectrum of $\LL$ coincides with $\{z_n\}$
\textup(counting multiplicities\textup).
\end{corollary}

\begin{proof}
Without loss of generality ${\rm Ker}\, \LL= 0$ and $\varkappa\ne 0$, where $\varkappa$ is given by \eqref{kap}.
In the case when $b_n \ne 0$ for any $n$ the Riesz basis property follows  from Theorem \ref{bas3}.
Eigenvalues of $\LL$ are zeros of the function $G$ given by \eqref{jjk}. By Proposition \ref{hry}, 
there exists an enumeration of the zero set of $G$ 
(counting multiplicities), $\ZZ(G) = \{z_n\}$, such 
that $|z_n -t_n| \asymp |a_nb_n|$, whence $\sum_n |z_n-t_n| <\infty$.
Conversely, by Proposition \ref{hry} and Lemma \ref{bbb}, any sequence $\{z_n\}$ such that 
$\sum_n |z_n-t_n| <\infty$ is a zero set of a function of the form
$$
G(z)  = A(z) \bigg( 1+ \sum_n \frac{c_n}{z- t_n}\bigg)
$$
with $\sum_n |c_n| <\infty$. It remains to choose appropriate $a_n, b_n$.

In the case when $b_n=0$ for some $n$, we formally cannot apply the model from Theorem \ref{model}. 
This case is however easily reduced to the cyclic case. 
Let $\mathbb{N} = N_1\cup N_2$, where $N_2 = \{ n:\ b_n=0\}$. 
Then we can write $\A = \A_1\oplus \A_2$, where $\A_j$ is multiplication by $z$ in $H_j = L^2(\nu_j) \cong\ell^2(N_j)$, 
$\nu_j = \sum_{n\in N_j} \delta_{s_n}$, $j=1,2$.
Now let $\LL =\A+a\otimes b$. With respect to decomposition 
$H = H_1\oplus H_2$ we have $a=a^{(1)}\oplus a^{(2)}$, $b=b^{(1)} \oplus 0$. 

Denote by $(e_m)_{m\in N_2}$ the standard orthogonal basis of $H_2$, 
$(e_m)_n = 0$, $m\ne n$, and $(e_m)_m=1$. It is clear that 
$0\oplus e_m$ is an eigenvector
of $\LL$ corresponding to the eigenvalue $t_m$, $m\in N_2$. 

Consider the operator $\LL_1 = \A_1 + a^{(1)} \otimes b^{(1)}$ on $H_1$. Now $b^{(1)}$
is a cyclic vector for $\A_1$. Its model space is $\mathcal{H}(T_1, A_1, \mu^{(1)})$,
where $T_1=\{t_n\}_{n\in N_1}$, $\mu^{(1)} = \sum_{n\in N_1} \mu_n \delta_{t_n}$,
and $\mu_n = |t_n|^2|b_n|^2$, $n\in N_1$. The corresponding ``characteristic'' function $G_1$ 
is given by
$$
G_1(z)  = \frac{1}{\varkappa} A_1(z) \bigg(1 - \sum_n \frac{a_n \bar b_n}{z- t_n}\bigg) = 
\frac{1}{\varkappa} A_1(z) \bigg(1 - \sum_{n\in N_1} \frac{a_n \bar b_n}{z- t_n}\bigg).
$$
The vector $b^1$ is cyclic and so, by Theorem \ref{bas3},
the normalized eigenvectors $(f_k)$ of $\LL_1$ form a Riesz basis in $H_1$ (to avoid uninteresting technicalities we assume that
all eigenvalues $\lambda_k$ are simple). Note that $\lambda_k$ (the zeros of $G_1$) 
are exactly the zeros of the function \eqref{jjk} which are different from $t_m$, $m\in N_2$.

For $u=u^{(1)}\oplus u^{(2)}$ we have
$\LL u = 
\LL_1 u^{(1)} \oplus \big(\A_2 u^{(2)} + (u^{(1)}, b^{(1)})a^{(2)}\big)$.
If $u$ is an eigenvector of $\LL$ then either $u^{(1)} =0$ (and so $u=0 \oplus e_m$ for some $m$) or $u^{(1)}\ne 0$ and 
then we can assume that $u^{(1)}= f_k$ for some $k$. In this case for $u^{(2)}$ we have the equation:
$$
\A_2 u^{(2)} + (f_k, b^{(1)})a^{(2)} = \lambda_k u^{(2)},
$$
whence $u^{(2)} = - (f_k, b^{(1)}) (\A_2 - \lambda_k I)^{-1} a^{(2)}$.
Note that any $\lambda_k$ is not in the spectrum of $\A_2$ which coincides with $\{t_m\}_{m\in N_2}$.

Thus, the set of eigenvectors of $\LL$ is of the form $\{f_k \oplus g_k\}_{k\in N_1} \cup \{0 \oplus e_m\}_{m\in N_2}$, 
where  $g_k = - (f_k, b^{(1)}) (\A_2 - \lambda_k I)^{-1} a^{(2)}$. 
Recall that a system  $\{h_n\}$ is a Riesz basis in a Hilbert space $H$ if and only if
$\{h_n\}$ is complete and $\|\sum_n c_n h_n\| \asymp \|(c_n)\|_{\ell^2}$, $(c_n)\in\ell^2$.
It is obvious that the eigenvectors of $\LL$ are complete. Since $\{f_k\}$ and $\{e_m\}$ 
are Riesz bases in $H_1$ and $H_2$ respectively, it is sufficient to show
that $\|\sum_k c_k g_k\| \lesssim \|(c_k)\|_{\ell^2}$, $(c_k)\in\ell^2$. This is, obviously, true, since
$$
(g_k)_j = \frac{a_j^{(2)}}{\lambda_k - t_j}(f_k, b^{(1)}),
$$
the sequence $\{(f_k, b^{(1)})\}$ is in $\ell^2$, while $\lambda_k-t_k \to 0$ 
and so the set $\{\lambda_k\}_{k\in N_1} \cup \{t_m\}_{m\in N_2}$\
is separated. 
\end{proof}

\begin{remark}
{\rm 1. Theorem \ref{bas4} extends \cite[Theorems 3.1, 4.1]{hry2} characterizing 
spectra of rank one perturbations of unbounded selfadjoint operators to the case of normal 
\smallskip
operators. 

2. There is a vast literature on Riesz bases of eigenvectors and root vectors of perturbations
of selfadjoint or normal operators with discrete spectrum;
for a detailed survey we refer to \cite{shkal}. It is possible
that our Theorem \ref{bas3} is covered by some more general results.  However, it seems that most of the results on Riesz bases 
of eigenvectors (see, e.g., \cite[Theorems 6.6, 7.1]{shkal})
apply to the case when the unperturbed operator is selfadjoint or normal with the spectrum on a finite system of rays. 
While we look at a very special class of perturbations (i.e., rank one), we do not need the assumptions on the spectrum location.
\smallskip

3. We give a more detailed comparison of
Theorem  \ref{bas3} with \cite[Theorem 7.1]{shkal} (which in a slightly weaker form goes back to \cite{mit}).
In this theorem perturbations $T+B$ of an unbounded selfadjoint operator $T$
with discrete spectrum  are considered such that the following ``local $p$-domination property'' holds:
for the normalized eigenvectors $\phi_n$ of $T$ and $p\in (-\infty, 1)$ one has $\|B \phi_n\| \lesssim |t_n|^p$.
In our situation the local $p$-domination condition is equivalent to $|b_n| \lesssim |t_n|^{p}$,
while the relation between $p$ and the order of the operator from \cite[Theorem 7.1]{shkal}  
means that $p\le -N$ where $N$ is the exponent of power separation
(note that the case $p\ge 0$ corresponds only to operators of order at least 1). 
In \cite{shkal} a somewhat weaker mean separation is used and so the systems of eigenvectors form Riesz bases with brackets.  Thus, Theorem 7.1 from \cite{shkal} says that a rank one perturbation
$\LL = \A+a\otimes b$ of a {\it selfadjoint} operator $\A$
has a Riesz basis of eigenvectors and root vectors if $(a_n) \in\ell^2$ and $(t_n^N b_n) \in \ell^\infty$.
Note also that $\A$ is assumed to be of order greater than $1/2$ which means that 
$N<2$. Theorem \ref{bas3} shows that for rank one perturbations conditions
$(a_n) \in\ell^2$ and $(t_n^N b_n) \in \ell^\infty$ are sufficient for a Riesz basis property
for any normal operator with spectrum in a finite union of strips and for any positive order. 
Under slightly stronger condition $(t_n^N b_n) \in \ell^q$ for some $q\in(0, \infty)$
the result remains true for any normal operator with power separated spectrum without 
any restrictions on its location.}
\end{remark}
\bigskip


\section{Two examples}
\label{examp}

We return to the problem about Riesz bases corresponding to zeros of the functions $B_\gamma$  of the form \eqref{gh}
and \eqref{ghh}. Our first example deals with the case of de Branges spaces (i.e., $T \subset \R$) while the second is 
a specific example of a CdB space corresponding to points on a cross  $\R\times i\R$ and which can be thought of as 
a cross counterpart of the Paley--Wiener space. 

Assume that $T \subset \R$. Then $\hht = \he$ for some Hermite--Biehler class function $E$. 
Moreover, we have $E=A-iB$, where
$$
B(z) = A(z) \bigg(q+ \sum_n \Big(\frac{1}{t_n-z} -\frac{1}{t_n}\Big) \mu_n \bigg)
$$
for some $q \in\R$. As usual we agree that the term in the brackets is simply $-1/z$ in case when $t_n=0$.
Note that the ``mass at infinity'' $p$ (from the representation \eqref{herg}) is zero since 
we assume that $\{K_{t_n}\}$ is an orthogonal basis in $\he$. 
The function $\beta = B/A$ is a Herglotz function in $\cp$.

By the de Branges theory we know that for all $\gamma\in \R$ (except, maybe, one) the zeros
of the functions
$$
B_\gamma(z) = A(z) \bigg(\gamma +q+ \sum_n \Big(\frac{1}{t_n-z} -\frac{1}{t_n}\Big) \mu_n \bigg) =B+\gamma A
$$
generate an orthogonal basis in $\hht=\he$. (Note a  slight difference in the definition of $B_\gamma$ -- now the parameter $q$
is included into $B_0$.) What happens if $\gamma\notin \R$?

Consider the inner function $\Theta = E^*/E = \frac{i-\beta}{i+\beta}$ in $\cp$. 
Since $\Theta$ is meromorphic in $\co$ it is of the form  
$\Theta(z) = e^{iaz} D(z)$, where $a\ge 0$ and $D$ is a meromorphic Blaschke product.
We have 
$$
\frac{B+\gamma A}{E} = \frac{\Theta-1}{2i} +\gamma \frac{\Theta+1}{2} = \frac{1}{2} (\gamma+i + (\gamma-i)\Theta).
$$
Thus, the zeros of $B+\gamma A$ are the solutions of the equation $\Theta = \frac{i+\gamma}{i-\gamma}$. In the case
$\gamma=i$ these are the poles of $\Theta$, i.e., the zeros of $E$. If $\gamma\in \cm$, then all zeros of $B+\gamma A$ are in $\cp$
and vise versa. 

Recall that a Blaschke product $D$ in the upper half-plane with zeros $\{z_n\}_{n\ge 1}$ is said to be interpolating 
if $\inf_n \prod_{k\ne n}\big|\frac{z_n-z_k}{z_n-\bar z_k}\big| >0$. 
By the classical Douglas--Shapiro--Shields theorem (see, e.g., \cite[Lecture VII]{nik}), 
the normalized Cauchy kernels $\{\tilde k_w\}_{w\in \ZZ(D)}$ form a Riesz basis in their closed linear span
$K_D = H^2 \ominus D H^2$ if and only if
$D$ is an interpolating Blaschke product. Here $k_w(z) = (1-\bar w z)^{-1}$ and
$\tilde k_w = (1-|w|^2)^{1/2} k_w$.

For $|\alpha|<1$ denote by $\Theta_\alpha$ the Frostman shift of $\Theta$,
$\Theta_\alpha = \frac{\Theta-\alpha}{1-\overline{\alpha}\Theta}$. 

\begin{proposition}
\label{carl}
Let $\hht  =\he$ be a de Branges space and let $A$, $B$, $B_\gamma$, $\Theta$ be as above. 
Let $\gamma\notin \R$ and put $\alpha = \frac{i+\gamma}{i-\gamma}$.
Then the following are equivalent:
\smallskip

1. The zeros of $B_\gamma$ are simple and the family $\{K_w\}_{w\in \ZZ(B_\gamma)}$
is a Riesz basis in $\he$;
\smallskip

2. $\gamma \in \cm$ and $\Theta_\alpha$ is an interpolating Blaschke product or 
$\gamma\in \cp$ and $\Theta_{1/\alpha}$  is an interpolating Blaschke product.
\end{proposition}

\begin{proof}
Assume that $|\alpha|<1$ and consider the 
normalized reproducing kernels of the model space $K_\Theta = E^{-1} \he$ at the points $w$ such that $\Theta(w) =\alpha$:
$$
\tilde k_w^{\Theta}(z)  = \Big(\frac{1-|w|^2}{1-|\Theta(w)|^2}\Big)^{1/2} \frac{1-\overline{\Theta(w)}\Theta(z)}{1-\bar w z} = 
\Big(\frac{1-|w|^2}{1-|\alpha|^2}\Big)^{1/2}  \frac{1-\overline{\alpha}\Theta(z)}{1-\bar w z}.
$$ 
A well-known transform (sometimes referred to as Crofoot's transform) 
$$
J: f\mapsto (1-|\alpha|^2)^{1/2} \frac{f}{1-\overline{\alpha}\Theta}
$$
maps $K_\Theta$ unitarily on the model space $K_{\Theta_\alpha}$. 
Clearly, $J$ maps the family $\{\tilde k_w^{\Theta}\}_{w\in \ZZ(\Theta-\alpha)}$ to 
$\{\tilde k_w\}_{w\in \ZZ(\Theta_\alpha)}$. Thus, we conclude that
$\{\tilde k_w^{\Theta}\}_{w\in \ZZ(\Theta-\alpha)}$ is a Riesz basis in $K_\Theta$ if and only if 
$\{\tilde k_w\}_{w\in \ZZ(\Theta_\alpha)}$ is a Riesz basis in $K_{\Theta_\alpha}$ if and only if
$\Theta_\alpha$ is an interpolating Blaschke product.  It remains to note that $\{\tilde K_w\}_{w\in \ZZ(B_\gamma)}$
 is a Riesz basis in $\he$ if and only if
$\{\tilde k_w^{\Theta}\}_{w\in \ZZ(\Theta-\alpha)}$ is a Riesz basis in $K_\Theta$ (see formula \eqref{repr}). 

In the case when $\gamma\in \cp$ and $|\alpha|>1$ note that the systems
 $\{\tilde K_w\}_{w\in \ZZ(B_\gamma)}$ and  $\{\tilde K_{\overline{w}}\}_{w\in \ZZ(B_\gamma)}$
simultaneously form a Riesz basis in $\he$ or not.
\end{proof}

\begin{remark}
{\rm Given an interpolating Blaschke product $\Theta$, it is a difficult problem to describe the set of parameters $\alpha$
for which $\Theta_\alpha$ also is interpolating. A. Nicolau \cite{nic1} gave a description of Carleson--Newman (e.g., finite product of interpolating) Blaschke products for which any Frostman shift is also a Carleson--Newman Blaschke product. He also gave
an example, for any $m\in (0,1)$, of an interpolating Blaschke product whose shifts with $|\alpha| \ge m$ are not Carleson--Newman.
On the other hand, if $\Theta$ is a {\it thin} Blaschke product (i.e., $ \prod_{k\ne n}\big|\frac{z_n-z_k}{z_n-\bar z_k}\big| \to 1$,
$n\to \infty$),
then all its Frostman shifts are also thin and thus interpolating up to possible gluing of a finite number of zeros. See \cite{nic2, nic3}
for further results and references therein. }
\end{remark}

Let $T = \Z +\frac{1}{2}$, $A(z) =\frac{1}{\pi} \cos (\pi z)$ and $\mu_n =1$ for any $t_n\in T$.
Then the space $\hh(T, A, \mu)$ coincides isometrically with the Paley--Wiener space $PW_\pi$. 
It is well known and easy to see that the systems of reproducing kernels corresponding to the zeros of  
$$
B_\gamma(z) = A(z) \bigg(\gamma + \sum_{n\in \Z} \Big(\frac{1}{n+\frac{1}{2}-z} - \frac{1}{n+\frac{1}{2}} \Big)\bigg) = 
\frac{\gamma}{\pi} \cos \pi z  + \sin\pi z
$$ 
form a Riesz basis in $\hht$ unless $\gamma = \pm \pi i$. In the case when $\gamma = \pm \pi i$ the function 
$B_\gamma(z) = \pm i \exp(\mp i \pi z)$
does not vanish and so the corresponding system is empty. 

Now we consider a certain cross analogue of the Paley--Wiener space.
Let 
\begin{equation}
\label{comlos}
T = \Big(\Z +\frac{1}{2}\Big) \cup i \Big(\Z +\frac{1}{2}\Big), \qquad A(z) =\frac{1}{\pi} \cos (\pi z) \cos (\pi i z)
\end{equation}
and $\mu_n =1$ for any $t_n\in T$. Then
$$
\frac{B_\gamma(z)}{A(z)} = \gamma + \sum_{t_n\in T} \bigg(\frac{1}{t_n-z} - \frac{1}{t_n}\bigg) =
\gamma +\pi \, \tg (\pi z) + \pi i\, \tg (\pi i z).
$$
It is easy to see that for any $\gamma$ all zeros of $B_\gamma$ (except possibly a finite number) 
are simple (see the discussion below). 

\begin{proposition}
\label{tangens}
Let $\hht$ be a CdB space with parameters \eqref{comlos}. Then the system of normalized kernels
$\{\tilde K_w\}_{w\in \ZZ(B_\gamma)}$ \textup(with added reproducing kernels in case of multiple zeros\textup) is complete 
for any $\gamma \in \co$, and is a Riesz basis in 
$\hht$ if and only if $\gamma \notin \{\pm \pi(1+i),\, \pm \pi(1-i)  \}$.
\end{proposition}

\begin{proof}
For the proof of the Riesz basis property we use the classical Bari's theorem (see, e.g., \cite[Lecture VI]{nik}): a normalized 
system $\{x_n\}$ in a Hilbert space $H$ 
is a Riesz basis if and only if:

a) $\{x_n\}$ is complete and minimal;

b) its biorthogonal system $\{y_n\}$ its complete; 

c) the map $J: x\mapsto (x, x_n)$ is a bounded linear map from $H$ to $\ell^2$; 

d) the map $\tilde J: x\mapsto (x, y_n)$ is a bounded linear map from $H$ to $\ell^2$.
\medskip
\\
{\it Step  1. Completeness of $\{\tilde K_w\}_{w\in \ZZ(B_\gamma)}$.}
Note that
\begin{equation}
\label{comlos1}
\frac{B_\gamma(z)}{A(z)} = \gamma +\frac{\pi}{i} \frac{e^{2\pi i z} - 1}{e^{2\pi i z} +1}
- \pi \frac{e^{2\pi  z} - 1}{e^{2\pi  z} +1}. 
\end{equation}
Then for any $\delta>0$ the ratio $B_\gamma/A$ tends respectively to 
$\gamma +\pi(-1+i)$, $\gamma +\pi(1+i)$, $\gamma +\pi(1-i)$, and 
$\gamma +\pi(-1-i)$, when $|z|\to \infty$ in each of the angles $\{\delta <\arg z<\pi/2 - \delta\}$,
 $\{\pi/2+\delta <\arg z<\pi - \delta\}$, $\{\pi+\delta <\arg z<3\pi/2 - \delta\}$, 
and $\{3\pi/2+\delta <\arg z<2\pi - \delta\}$. Now, if $f = A\sum_{t_n\in T} \frac{c_n}{z-t_n} \in \hht$ is orthogonal 
to $\{\tilde K_w\}_{w\in \ZZ(B_\gamma)}$, then we can write $f=B_\gamma h$ for some entire function $h$.
Since $f/A$ tends to zero in each of the above four angles, we conclude that $h$ tends to zero at least in three of them. 
Since $h$ must be a function of order at most 1, we conclude that it is zero. 
\medskip
\\
{\it Step 2. Completeness of the biorthogonal system.} It should be noted that if the system 
$\{K_w\}_{w\in \ZZ(B_\gamma)}$ is complete in some CdB space $\hht$, then its biorthogonal
is always complete. We do not use here the special form of $\hht$. Indeed, for any $w$ such that
$B_\gamma(w) = 0$ we have
$$
\frac{B_\gamma(z)}{z-w} = A(z)\sum_n \frac{\mu_n}{(w-t_n)(z-t_n)} \in \hht.
$$

From now on, to avoid uninteresting technicalities, we assume that all zeros of $B_\gamma$ are simple. 
Then the biorthogonal system to $\{K_w\}_{w\in \ZZ(B_\gamma)}$  
is given  by $\big\{ \frac{B_\gamma(z)}{B'_\gamma(w)(z-w)} \big\}$. Assume that $f=A\sum_n \frac{c_n \mu_n^{1/2}}{z-t_n}$
is orthogonal to this system. Hence, for any $w\in \ZZ(B_\gamma)$ one has
$$
\sum_n \frac{\bar c_n \mu_n^{1/2}}{w-t_n} =0.
$$
Therefore, the function $f^*$ defined by $f^*(z) = A(z) \sum_n \frac{\bar c_n \mu_n^{1/2}}{z-t_n}$ belongs to $\hht$
and vanishes on $\ZZ(B_\gamma)$, a contradiction with completeness of
$\{K_w\}_{w\in \ZZ(B_\gamma)}$ unless $c_n=0$ for any $n$.
\medskip
\\
{\it Step 3. Zeros of $B_\gamma$ when $\gamma \notin \{\pm \pi(1+i),\, \pm \pi(1-i)\}$.}
Analyzing the equation $B_\gamma = 0$ using the representation \eqref{comlos1} and the Rouch\'e theorem 
it is easy to see that in the case
$\gamma \notin \{\pm \pi(1+i),\, \pm \pi(1-i)\}$ its solutions consist of four series of zeros. Namely there exist 
$\alpha_j, \beta_j \in \R$ (depending on $\gamma$),
$j=1, \dots, 4$, 
such that all sufficiently large zeros of $B_\gamma$ are of one of the following form:
\begin{equation}
\label{comlos2}
\begin{aligned}
& k + \alpha_1 +i\beta_1 + \delta_{1k}, \quad ik + i\alpha_2 -\beta_2 + \delta_{2k}, \\
-& k + \alpha_3 +i\beta_3 + \delta_{3k},
\quad -ik + i\alpha_4 -\beta_4 + \delta_{4k},
\end{aligned}
\end{equation}
where $k\in \N$ is sufficiently large and $\delta_{jk} = O(e^{-2\pi k})$.
Note also that for any $\gamma\in\co$ we have $\ZZ(B_\gamma) \cap T = \emptyset$
and, moreover, $\alpha_j +i\beta_j \notin \Z+\frac{1}{2}$.
\medskip
\\
{\it Step 4. Boundedness of $J$ and $\tilde J$.}
Recall that $\|K_w\|^2 = |A(w)|^2 \sum_n \frac{\mu_n}{|t_n-w|^2}$. 
From the representations \eqref{comlos2} it follows that, for any fixed
$\gamma \notin \{\pm \pi(1+i),\, \pm \pi(1-i)\}$ one has $\|K_w\| \asymp |A(w)|$, 
$w\in \ZZ(B_\gamma)$. Since, for $f=A\sum_n \frac{c_n}{z-t_n} \in \hht$,
$$
(f, \tilde K_w) = \frac{A(w)}{\|K_w\|} \sum_n \frac{c_n}{w-t_n},
$$
we see that the boundedness of $J$ is equivalent to the boundedness of the operator
\begin{equation}
\label{romen}
(c_n) \in\ell^2 \mapsto \Big(\sum_n \frac{c_n}{w-t_n}\Big)_{w\in \ZZ(B_\gamma)}.
\end{equation}
In view of the asymptotic representation \eqref{comlos2} of $\ZZ(B_\gamma)$ this operator is a small perturbation 
of discrete Hilbert and Hankel-type operators $(c_n)_{n\in \Z} \mapsto \big(\sum_{n\ne k} \frac{c_n}{n-k}\big)_{k\in \Z}$ and 
$(c_n)_{n\in \Z} \mapsto \big(\sum_{n\ne 0} \frac{c_n}{n+i k}\big)_{k\in \Z}$, and thus is bounded. 

The biorthogonal system to $\{\tilde K_w\}_{w\in \ZZ(B_\gamma)}$ is given by 
$\big\{\frac{\|K_w\| B_\gamma(z)}{B'_\gamma(w)(z-w)} \big\}$. Then, for
$f=A\sum_n \frac{c_n}{z-t_n} \in \hht$, we have (by \eqref{skal}) 
$$
\Big(\frac{\|K_w\| B_\gamma(z)}{B'_\gamma(w)(z-w)}, f \Big) = \sum_n 
\frac{\|K_w\| B_\gamma(t_n)}{A'(t_n) B'_\gamma(w)}\cdot \frac{\bar c_n}{t_n-w}. 
$$
It is easy to see that $|A'(t_n)| \asymp |B_\gamma(t_n)|$ and,
from the asymptotics \eqref{comlos2}, that  $\|K_w\| \asymp |A(w)| \asymp |B'_\gamma(w)|$, $w\in \ZZ(B_\gamma)$.
Thus, the boundedness  of $\tilde J$ is also equivalent to the boundedness of the operator \eqref{romen}.
\medskip
\\
{\it Step 5. Case of $\gamma \in \{\pm \pi(1+i),\, \pm \pi(1-i)\}$.}
In this case two of series of zeros \eqref{comlos2} glue together. Let us consider the case
$\gamma = \pi(1-i)$. Then it is easy to see that for all sufficiently large zeros in the first quadrant
are of the form
$$
\Big(\frac{k}{2} - \frac{1}{8}\Big)(1+i) +O(e^{-2\pi k})
$$
where $k\in \N$ is sufficiently large. We also have two series of zeros
$$
-k + \frac{5}{8} +\frac{i}{4\pi} \log 2, \qquad -ki + \frac{5}{8}i +\frac{1}{4\pi} \log 2,
$$
as $k\to \infty$.

It is easy to see that for the zeros $w_k = \big(\frac{k}{2} - \frac{1}{8}\big)(1+i) +O(e^{-2\pi k})$ (with sufficiently large $k$)
one has 
$$
\|K_{w_k}\|^2 = |A(w_k)|^2 \sum_n \frac{\mu_n}{|t_n-w_k|^2} \asymp
|A(w_k)|^2 \sum_{n=1}^\infty \frac{1}{(k-n)^2 +k^2}\asymp \frac{|A(w_k)|^2}{k}.
$$
Take $f(z) = \frac{A(z)}{z- 1/2}$. 
Then $|(f, \tilde K_{w_k})| \asymp \frac{k^{1/2}}{|w_k - 1/2|}$, $k\to \infty$,
whence $\big( (f, \tilde K_{w_k}) \big) \notin\ell^2$ and $J$ is trivially unbounded. 

The same is true for other exceptional values of $\gamma$.
\end{proof}


\section{Nearly invariant subspaces of finite codimension}
\label{riv}

In this section we prove  Theorem \ref{dom}.

\begin{proof}[Proof of Theorem \ref{dom}]
{\bf (iii) $\Longrightarrow$ (ii).} Assume that $\sum_n |t_n|^{2N-2}\mu_n <\infty$. Then
the functions $B_j$ given by \eqref{bj} (i.e., $B_j  \longleftrightarrow (t_n^j \mu_n^{1/2})$) belong to $\hht$ for $j=0, \dots, N-1$.
Let us show that $B_j \perp {\rm clos}\,\mathcal{D}_{z^N}$. Let
$f\in \mathcal{D}_{z^N}$, $f \longleftrightarrow (c_n)$. Then $z^k f  \longleftrightarrow (t_n^k c_n)$, $1\le k\le N$.
 
Assume that $0\notin T$ (the general case follows trivially by a shift). 
Then, for $j=0, \dots, N-1$,
$$
(f, B_j) = \sum_n c_n t_n^j \mu_n^{1/2} = \bigg(\sum_n \frac{c_n t^{j+1}\mu_n^{1/2}}{t_n-z}\bigg) \bigg|_{z=0} = 
-\frac{z^{j+1}f(z)}{A(z)}\bigg|_{z=0} = 0.
$$
The functions $B_j$ are linearly independent and so the codimension of ${\rm clos}\,\mathcal{D}_{z^N}$ 
is at least $N$. To show that its codimension is at most $N$, assume that 
$g\in \hht$, $g \longleftrightarrow (d_n)$ and $g\perp \mathcal{D}_{z^N}$. 
We already know that $B_j \perp {\rm clos}\,\mathcal{D}_{z^N}$, $j=0, \dots, N-1$. Subtracting from $g$
a linear combination of $B_j$ we can assume that $d_1 = \dots =d_N=0$.
Next, for any $m>N$, one has
$$
g\perp \frac{A(z)}{(z-t_m)\prod_{k=1}^N (z-t_k)},
$$
whence
$$
d_m \mu_m^{-1/2} \prod_{k=1}^N (t_m-t_k)^{-1} +\sum_{k=1}^N 
d_k \mu_k^{-1/2} (t_k-t_m)  \prod_{1\le l\le N, l\ne k} (t_k-t_l)^{-1} = 0.
$$
We concude that $d_m =0$ for $m>N$, whence $g=0$. Thus, 
$(\mathcal{D}_{z^N})^\perp = {\rm Span}\, \{B_0, \dots, B_{N-1}\}$.
\bigskip
\\
{\bf (ii) $\Longrightarrow$ (i).} All we need is to show
that ${\rm clos}\,\mathcal{D}_{z^N}$ is nearly invariant. Let $f\in {\rm clos}\,\mathcal{D}_{z^N}$  and $f(\lambda) =0$.
Choose $f_n, g \in  \mathcal{D}_{z^N}$ such that $f_n\to f$ and  $g(\lambda)=1$. Then
$$
\frac{f_n - f_n(\lambda) g}{z-\lambda} \to \frac{f}{z-\lambda}.
$$
We use the fact that $f_n(\lambda) \to f(\lambda)$ and that the map $h \to \frac{h}{z-\lambda}$ is a bounded operator from the subspace
$\{h\in\hht: \ h(\lambda) =0\}$ to $\hht$. 
\medskip

While the proof essentially goes by (iii) $\Longrightarrow$ (ii) $\Longrightarrow$ (i) $\Longrightarrow$ (iii), we also 
discuss the implication (iii) $\Longrightarrow$ (i) since its conclusion will be needed later on. 
\bigskip
\\
{\bf (iii) $\Longrightarrow$ (i).} Define $\hh_0 = \{B_0, \dots, B_{N-1}\}^\perp$. We need to show that
it is nearly invariant, i.e., if $f\perp B_0, \dots, B_{N-1} $ and $f(\lambda) = 0$, then $\frac{f(z)}{z-\lambda} \perp B_0, \dots, B_{N-1}$.
We have, for $\lambda\notin T$,
$$
\begin{aligned}
\Big(\frac{f}{z-\lambda}, B_j \Big)  = \sum_n \frac{c_n t_n^j \mu_n^{1/2}}{t_n-\lambda} & = 
\sum_n \frac{c_n (t_n^j -\lambda^j) \mu_n^{1/2}}{t_n-\lambda} \\
& =\sum_{l=0}^{j-1} \lambda^{j-1-l} \sum_n c_n 
t_n^l \mu_n^{1/2} =0.
\end{aligned}
$$
The case when $\lambda\in T$ is analogous. 
\bigskip
\\
{\bf (i) $\Longrightarrow$ (iii).} This implication is slightly more complicated. Assume that $\hh_0$ is of codimension $N$
and let $\hh_0^\perp = {\rm Span}\, \{g_1, \dots, g_{N}\}$. Let $g_j \longleftrightarrow (d_n^j)$. Since $\hh_0$
is nearly invariant we see that for any $f\perp g_1, \dots, g_N$ such that $f(\lambda) = 0$ 
we have $\frac{f}{z-\lambda} \perp g_1, \dots, g_N $. Equivalently, this means that 
$$
\sum_n c_n \overline{d_n^j} = 0, \ \ j=1, \dots, N, \quad \sum_n \frac{c_n \mu_n^{1/2}}{t_n - \lambda} =0 \quad \Longrightarrow
\quad 
\sum_n \frac{c_n \overline{d_n^j}}{t_n-\lambda} = 0, 
$$
that is, for any $k$,
$$
\Big(\frac{d_n^k}{\bar t_n - \bar \lambda} \Big)  \in 
 {\rm Span}\,\bigg\{ \Big(\frac{\mu_n^{1/2}}{\bar t_n - \bar \lambda} \Big),   (d_n^j), \ 1\le j\le N \bigg\}.
$$
Thus, there exists a matrix $\Gamma = (\gamma_{kj})_{1\le k,j\le N}$ and $\beta_k$, $1\le k\le N$, such that
for any $n$
$$
\frac{d_n^k}{\bar t_n - \bar \lambda} = \sum_{j=1}^N \gamma_{kj} d_n^j + \beta_k 
\frac{\mu_n^{1/2}}{\bar t_n - \bar \lambda}, \qquad 1\le k\le N.
$$
If we denote by $d_n$ the vector (column) $(d_n^j)_{j=1}^N$, and put $\beta = (\beta_k)_{k=1}^N$,
the equation takes form 
$$
(I- (\bar t_n - \bar \lambda) \Gamma) d_n =  \mu_n^{1/2} \beta. 
$$
Note that $\Gamma$ and $\beta$ depend on $\lambda$, but do not depend on $n$. From now on we assume that 
$\lambda =0$ and $0 \notin T$. Then the equation takes the form 
\begin{equation}
\label{mmd}
(I- \bar t_n \Gamma) d_n =  \mu_n^{1/2} \beta. 
\end{equation}

The solutions of the equation $(I-z\Gamma) u = v$ in $\co^N$ are rational functions of $z$ whose poles are the zeros of
${\rm det}\, (I-z\Gamma)$ or, equivalently, $z^{-1}$ where $z$ is an eigenvalue of $\Gamma$. 
Since we know that equation \eqref{mmd} has the solution, we conclude that for the values of the parameter
$z= \bar t_n $ and
for the right hand side $\mu_n^{1/2} \beta$ either there are no singularities or they cancel. However, 
we cannot apriory exclude that $\bar t_n^{-1}$ is an eigenvalue of $\Gamma$. 
Therefore, for such $n$ the solution is not unique. 
Denote by $\nn$ the (obviously finite) set of indices $n$ such 
that $\bar t_n^{-1}$ is an eigenvector of $\Gamma$. 
To summarize, the vectors $d_n = (d_n^j)_{j=1}^N$ 
satisfying \eqref{mmd} will be necessarily of the form 
$$
d_n^j = \mu_n^{1/2} R_j (\bar t_n) +u_n^j,
$$
where $R_j$ is a rational function independent on $n$, $R_j = P_j/Q_j$, where $P_j$, $Q_j$ are coprime polynomials such that
${\rm deg}\, P_j \le N-1$, ${\rm deg}\, Q_j \le N$ and $(u_n^j)_{n\ge 1}$ is a finite vector with possible nonzero
components only at $n\in \nn$. We now analyze these solutions in more detail.
\medskip
\\
{\it Step 1.} At this step we assume that $j$ is fixed. We claim that $Q_j$ is a constant. 
Assume that this is not the case and so $R_j$ has at least one pole $\bar \gamma_1$. 
Note that $\gamma_1\notin T$. Then we have 
\begin{equation}
\label{glo}
d_n^j = \mu_n^{1/2}\bigg(p_j (\bar t_n) +\sum_{l=1}^L \sum_{k=1}^{m_l} \frac{c_{lk}}
{(\bar t_n -\bar \gamma_l)^{k}}\bigg) 
+u_n^j,
\end{equation}
where $p_j$ is a polynomial of degree at most $N-1$, $\gamma_l$ are distinct numbers in $\co\setminus T$ and
$L$ is at least 1. We will obtain a contradiction with the fact that $\hh_0$ is nearly invariant, by producing 
a function $f\in \hh_0$ such that $f(\gamma_1) =0$, but $\frac{f(z)}{z-\gamma_1} \notin \hh_0$. 
Note that for $\gamma\notin T$ and $f\in \hht$, $f\longleftrightarrow (c_n)$,
one has 
$$
\bigg( (c_n), \Big(\frac{\mu_n^{1/2}}{(\bar t_n -\bar \gamma)^k}\Big)\bigg)_{\ell^2} =
\sum_n  \frac{\mu_n^{1/2} c_n }{(t_n - \gamma)^k} = -\frac{1}{(k-1)!}\Big(\frac{f}{A}(z)\Big)^{(k-1)}\Big|_{z=\gamma}.
$$
Also, if $m=\max_j {\rm deg}\, p_j $, then we can already conclude that $\sum_n  |t_n|^{2m} \mu_n <\infty$ 
(other terms in \eqref{glo} belong to $\ell^2$), whence  
$B_0, \dots, B_m\in \hht$. Since the subspace $\hh_0$ is nearly invariant (and therefore considering $\frac{z-\tilde \gamma}{z-\gamma}f$
we can move any zero $\gamma$ of $f$ to any other point $\tilde\gamma$) and infinite dimensional 
we can choose $f \in \hh_0$, $f\longleftrightarrow (c_n)$, such that:
\begin{itemize} 
\item $f \perp B_0, \dots, B_m$, whence $f\perp g$, where $g \longleftrightarrow (p_j(\bar t_n))$;
\item $c_n = 0$, $n\in \nn$, whence $f\perp g$, where $g \longleftrightarrow (u_n^j)$;
\item $f^{(k)}(\gamma_l) = 0$, $1\le l \le L$, $1\le k\le m_l-1$;
\item $f^{(m_1)}(\gamma_1) \ne 0$.
\end{itemize}

From the above conditions we see that $f\perp g_j$, $g_j \longleftrightarrow (d_n^j)$, and $f(\gamma_1) = 0$.
Recall that by implication (iii)$\Longrightarrow$(i) we have $\frac{f(z)}{z-\gamma_1} \perp B_0, \dots, B_m$
as well. Thus, the function $\frac{f(z)}{z-\gamma_1}$ is orthogonal to all terms in the representation \eqref{glo}
of $d_n^j$ except $\big(\frac{\mu_n^{1/2}}{(\bar t_n -\bar \gamma_1)^{m_1}}\big)$. Thus, 
$\big(\frac{f}{z-\gamma_1}, g_j\big) \ne 0$, a contradiction to nearly invariance of $\hh_0$.
\medskip
\\
{\it Step 2.}  By Step 1 we know that, {\it for any $j$}, $d_n^j = \mu_n^{1/2} p_j (\bar t_n) + u_n^j$, where $p_j$ is some polynomial
of degree at most $N-1$ and $(u_n^j)$ is a finite vector. We claim that
$u_n^j = 0$ for all $j$ and $n\in\nn$. Assume that there exist $j$ and $n_0$ such that $u_{n_0}^j \ne 0$.
As before, we can choose $f\in \hh_0$ such that $f \perp  B_0, \dots, B_m$,
$m=\max_j {\rm deg}\, p_j$, and such that $c_n = 0$, $n\in \nn$, which is equivalent to $f(t_n) = 0$, $n\in \nn$.
We can also choose $f$ so that $f'(t_{n_0}) \ne 0$. Hence, $\frac{f}{z-t_{n_0}}$ is not orthogonal to 
$(u_n^j)$, whence  $\big(\frac{f}{z-{t_{n_0}}}, g_j\big) \ne 0$, again a contradiction.
\medskip
\\
{\it Step 3.} We conclude that $(d_n^j) = (\mu_n^{1/2} p_j (\bar t_n))$ for any $j$
where $m_j = {\rm deg}\, p_j  \le N- 1$ for any $j$, $1\le j\le N$. If $m_j<N-1$  for any $j$, we conclude that $g_j$
are linearly dependent, a contradiction. Thus, there exists $m_j = N-1$, whence
$\sum_n  |t_n|^{2N-2} \mu_n <\infty$. Moreover, we see that
${\rm Span}\, \{g_1, \dots, g_N\} = {\rm Span}\, \{ B_0, \dots, B_{N-1}\}$.
Theorem \ref{dom} is proved.
\end{proof}

It is easy to see that the finite codimension subspaces constructed in Theorem \ref{dom}
are isomorphic to Cauchy--de Branges spaces. However, they are never CdB spaces themselves with the norm inherited from $\hht$
unless $\hht$ is a rotation of a de Branges space.
We analyze the case of subspaces of codimension 1.

\begin{proposition} 
\label{nobas}
Let $\hht$ be a small de Branges space and let $\hh_0 = {\rm clos}\,\mathcal{D}_{z} = \{B\}^\perp$, where $B=B_0$.
\medskip

1. Fix $t_0\in T$ and put $T_0 = T\setminus \{t_0\}$, $A_0(z) = \frac{A(z)}{z-t_0}$ and $\nu = \sum_{n\ne 0} |t_n|^2 \mu_n \delta_{t_n}$.
Then $\hh_0$ is isomorphic to $\hh(T_0, A_0, \nu)$.
\medskip

2. $\hh_0$ is itself a CdB space, that is, has an orthogonal basis of reproducing kernels, if and only if $T$ lies on a straight line 
\textup(i.e., $\hht$ is a rotation of a de Branges space\textup).
\end{proposition}

\begin{proof}
1. Let $T= \{t_n\}_{n\ge 0}$ and  let $f\in \hh_0$, $f  \longleftrightarrow (c_n)_{n\ge 0}$. Since $f\perp B$, one has $\mu_0^{1/2} c_0 = - \sum_{n\ge 1} c_n \mu_n^{1/2}$. Then
$$
f(z) = A(z) \sum_{n\ge 1} c_n \mu_n^{1/2} \bigg(\frac{1}{z-t_n}  - \frac{1}{z-t_0} \bigg) = 
\frac{A(z)}{z-t_0} \sum_{n\ge 1} \frac{c_n(t_n-t_0)\mu_n^{1/2}}{z-t_n}.
$$ 
Thus, $f\in\hh(T_0, A_0, \nu) $ and $\|f\|^2_{\hh(T_0, A_0, \nu)} \asymp \sum_{n\ge 1} |c_n|^2\asymp
\sum_{n\ge 0} |c_n|^2 =\|f\|^2_{\hh(T, A, \mu)}$.

Conversely, let $g\in \hh(T_0, A_0, \nu)$. Then, for some $(c_n)_{n\ge 1}\in\ell^2$ one has
$$
g(z) =A_0(z) \sum_{n\ge 1} \frac{c_n(t_n-t_0)\mu_n^{1/2}}{z-t_n}.
$$
It remains to put $c_0 = - \mu_0^{-1/2} \sum_{n\ge 1} c_n\mu_n^{1/2}$ and to reverse the above computations.
\medskip

2. Denote by $\hat K_w$ the reproducing kernel of $\hh_0$.
If we assume that $\|B\|^2_{\hht}  = \sum_n \mu_n =1$, then $\hat K_w  = K_w - \overline{B(w)}B $,
where $K_w$ is the reproducing kernel of $\hht$. 

Assume that $\hh_0$ has an orthogonal basis of reproducing kernels $\{\hat K_{s_l}\}$. 
Then $\hat K_{s_l}(s_m) =0 $, $m\ne l$. For any $m,l$, 
the function $\frac{z-s_l}{z-s_m} \hat K_{s_l}$ belongs to $\hh_0$ and is orthogonal to all $\hat K_{s_j}$, $j\ne m$. Therefore, 
there exists a constant $d_{l,m}$ such that
$$
\frac{z-s_l}{z-s_m} \hat K_{s_l} = d_{l,m} \hat K_{s_m}.
$$
Since $\hat K_w  = K_w - \overline{B(w)}B $, we obtain by comparing the values at $t_n$ (for  $t_n \ne s_l, s_m$) that
\begin{equation}
\label{bro}
\frac{t_n-s_l}{t_n -s_m} \bigg( \frac{\overline{A(s_l)}}{\bar s_l-\bar t_n}
- \overline{B(s_l)}\bigg) = d_{l,m}
\bigg( \frac{\overline{A(s_m)}}{\bar s_m-\bar t_n}
- \overline{B(s_m)} \bigg).
\end{equation}
Taking the limit as $n\to \infty$ we conclude that $\overline{B(s_l)} = d_{l,m} \overline{B(s_m)}$. 
\medskip
\\
{\it Case 1.} Assume that $B(s_m) =0$ for some $m$. Then $B(s_l) = 0$ for any $l$ and thus $\{s_k\} \subset \ZZ(B)$.
Then, for any $n$ such that  $t_n \ne s_l, s_m$,
$$
\overline{A(s_l)} \frac{t_n-s_l}{t_n -s_m}  = 
\overline{A(s_m)} \frac{\bar t_n - \bar s_l}{\bar t_n - \bar s_m}.
$$
Since $t_n\to \infty$, we conclude  that $A(s_l) = A(s_m)$, whence $t_n$ satisfy the equation
$\ima (t_n (\bar s_m -\bar s_l) + \bar s_l s_m) =0$, which is an equation of a line. 
\medskip
\\
{\it Case 2.} Now assume that $B(s_m) \ne 0$ for any $m$. Then $d_{l,m} = \overline{B(s_l)}/\overline{B(s_m)}$.
Inserting this into \eqref{bro} we get
$$
\overline{A(s_m)}\, \overline{B(s_l)} \,\frac{ t_n - s_m}{\bar t_n - \bar s_m}  - 
\overline{A(s_l)} \,\overline{B(s_m)}\, \frac{t_n-s_l}{\bar t_n - \bar s_l}
= \overline{B(s_l)}\, \overline{B(s_m)} (s_m-s_l)
$$
for $t_n \ne s_l, s_m$. 
This equation is equivalent to $a|t_n|^2 +b(\bar t_n)^2 +ct_n +d\bar t_n +e= 0$ for some
$a, \dots , e\in \co$, $b= \overline{B(s_l)}\, \overline{B(s_m)} (s_m-s_l) \ne 0$. Since $t_n \to \infty$, we have $|a| = |b|$.
Thus there exists $\theta\in\R$ such that $e^{i\theta}t_n$ satisfy the equation $|z|^2 - z^2 + az+b\bar z+c =0$
for some new coefficients $a,b,c\in \co$. It is easy to show that this equation has infinitely many solutions 
tending to infinity only if it is equivalent to an equation $\ima z=const$. 

Thus, in each of two cases all points $t_n$ (maybe, except two if $s_m\in T$ or $s_l\in T$) lie on the same line. 
Starting from different values of $s_m, s_l$ we conclude that all $t_n$
are on the same line.
\end{proof}

\begin{remark}
{\rm Statement 2 of Proposition \ref{nobas} and its proof are
close to Theorem \ref{jam} (\cite[Theorem 1]{bms1}) by Belov, Mengestie and Seip. 
It is also worth mentioning  that Theorem \ref{jam} can be related via the functional model of Section \ref{appli}
with a result of E.~Ionascu \cite{ion} who showed that the spectra of diagonal normal operators who have a normal
rank one perturbation must lie on a straight line or one a circle. }
\end{remark}
\bigskip


\section{Density of polynomials and structure of nearly invariant subspaces}
\label{stru}

Assume that $\sum_n \mu_n |t_n|^{2j} <\infty$ for any $j$. Then, by Theorem \ref{dom},
the space $\hht$ contains a subspace of any finite codimension and these subspaces are ordered
by inclusion. It turns out that, under additional condition that the set of all polynomials $\pp$ is dense
in $L^2(\mu)$, all nearly invariant subspaces of $\hht$ are of finite codimension; therefore they are 
of the form ${\rm clos}\,\mathcal{D}_{z^j}$, $j\in \N$, and are ordered by inclusion.

De Branges and Cauchy--de Branges spaces with ``small'' measures $\mu$ (i.e., with fast decay of $\mu_n$)
were studied in \cite{loc1, loc2}, where the notion of localization was introduced.
In what follows we will assume that the sequence $T$ is {\it power separated} (see \eqref{powsep}). 
Recall that any power separated sequence has finite convergence exponent and so $\hht$ is a space of finite order. 

We say that the space $\mathcal{H}(T,A,\mu)$
with a power separated sequence $T$:
\begin{itemize}
\item 
has the {\it localization property} if 
there exists a sequence of disjoint disks $\{D(t_n,r_n)\}_{t_n\in T}$ 
with $r_n\to 0$ such that for any nonzero $f\in \mathcal{H}(T,A,\mu)$ the set
$\mathcal{Z}(f)\setminus\cup_n D(t_n,r_n)$ is finite and each disk
$D(t_n,r_n)$ contains at most one point of $\mathcal{Z}(f)$ (counting multiplicities) for any
$n$ except, possibly, a finite number;
\medskip

\item 
has the {\it strong localization property} if there exists
a sequence of disjoint disks $\{D(t_n,r_n)\}_{t_n\in T}$ 
with $r_n\to 0$ such that for any nonzero $f\in \mathcal{H}(T,A,\mu)$ 
the set $\mathcal{Z}(f)\setminus\cup_n D(t_n,r_n)$ is finite
and each disk $D(t_n,r_n)$ contains
exactly one point of $\mathcal{Z}(f)$ for any $n$ except, possibly,
a finite number. 
\end{itemize}

As is shown \cite{loc1}, one can take $r_n = \delta (|t_n|+1)^{-M}$ for any $M>N$, where $N$ is the constant from the 
separation condition \eqref{powsep}. 

It was proved in \cite{loc1} in the de Brangean setting and in \cite{loc2} for general CdB-spaces
that the strong localization property is equivalent to the density of polynomials in 
$L^2(\mu)$, $\mu= \sum_n \mu_n \delta_{t_n}$ (the elements of $L^2(\mu)$ can be identified with sequences 
$(d_n)$ such that $\|(d_n)\|^2_{L^2(\mu)} = \sum_n |d_n|^2\mu_n <\infty$).
Therefore, we can add to Theorem \ref{stro1}
one more equivalent condition.

\begin{theorem}
\label{stro2}
Let $\htt$ be a CdB-space with a power separated sequence $T$. Then the following assertions are equivalent:
\medskip

1. $\htt$ contains a nearly invariant subspace of any finite codimension and any nontrivial 
nearly invariant subspace is of this form. 
\medskip

2. The set of all polynomials $\pp$ is contained in $L^2(\mu)$ and is dense there.
\medskip

3. The space $\mathcal{H}(T,A,\mu)$ has the strong localization property.
\end{theorem}

\begin{proof}
Equivalence of statements  2 and 3 was shown in \cite[Theorem 1.2]{loc2}. 
\medskip
\\
{\bf $1 \Longrightarrow 2$.} By Theorem \ref{dom}, $\sum_n \mu_n |t_n|^{2j} <\infty$ for any $j$, and so 
$\pp \subset L^2(\mu)$. Assume that the polynomials are not dense in $L^2(\mu)$. Then there exists
a nonzero sequence $(d_n)\in L^2(\mu)$ such that $\sum_n d_n t_n^j \mu_n = 0$. If we put 
$c_n  = d_n \mu_n^{1/2}$, then $(c_n) \in \ell^2$ and   $\sum_n c_n t_n^j \mu^{1/2}_n = 0$.
Therefore, the function $f(z) = A(z) \sum_n \frac{c_n \mu_n^{1/2}}{z-t_n}$ is in $\hht$ 
and is orthogonal to all functions $B_j$ from \eqref{bj}. Put $\hh_0 = \{B_j: j\ge 0\}^\perp$. 
Then $\hh_0$ is nontrivial, of infinite codimension and, by the proof of the implication (iii)$\Longrightarrow$(i) of Theorem \ref{dom}, 
$\hh_0$ is nearly invariant, a contradiction.
\medskip
\\
{\bf $3 \Longrightarrow 1$.}  Assume that $\hht$ has the strong localization property
and let $\hh_0$ be its nearly invariant subspace. Let $f\in\hh_0$, $f\ne 0$. Fix $M>N$,
where $N$ is the constant from the separation condition \eqref{powsep}. 
Assume that $T$ is enumerated by positive integers, $T = \{t_n\}_{n\ge 1}$.
Then there exists $\nn$ such that $\N \setminus \nn$ is finite and for any 
$n\in \nn$ the disk $D(t_n, (|t_n|+1)^{-M})$ contains exactly one zero of $f$, say $z_n$. Put
$$
g(z) = f(z)\prod_{n\in \nn} \frac{z-t_n}{z-z_n}, \qquad g_k(z) =  f(z)\prod_{n\in \nn, n\le k} \frac{z-t_n}{z-z_n}.
$$
It is easy to show that if $M$ is sufficiently large, e.g., $\sum_n |t_n|^{M-N} <\infty$, then $g\in\hht$ (see \cite[Section 3]{loc2} for details)
and moreover $g_k\to g$ in $\hht$. Since $g_k\in \hh_0$, we conclude that $g\in \hh_0$. Also, we can write
$$
g(z) = \frac{A(z)}{P(z)}h(z),
$$
where $P(z) = \prod_{n\in \N \setminus \nn}(z-t_n)$ is a polynomial and $h$ is some entire function.

Any function in $\hht$ is majorized by $A$ in the sense that $\frac{f(z)}{A(z)} = o(1)$
as $z\to\infty$ outside a set of zero area density. Thus, $h$ admits a polynomial estimate outside 
a set of zero area density, whence $h$ is a polynomial by Theorem \ref{dens}. 
We conclude that $\hh_0$ contains a function of the form $\frac{A}{P}$ where $P$
is a polynomial of some degree $m$ with zeros in $T$. Since $\hh_0$ is nearly invariant, 
we have $\frac{A(z)}{(z-t_1)\dots (z-t_m)} \in \hh_0$ for any choice of distinct $t_1, \dots, t_m$.

It remains to note that  for $F\in\hht$
$$
F\perp \Big\{ \frac{A(z)}{(z-t_1)\dots (z-t_m)}\Big\} \quad \Longleftrightarrow \quad F\in
{\rm Span}\, \{ B_0, \dots, B_{m-2}\}.
$$ 
Then ${\rm clos}\,\mathcal{D}_{z^{m-1}} \subset \hh_0$, whence the codimension of $\hh_0$ is finite. 
\end{proof}

Further results related, in particular, to a more subtle property of localization (which is not strong)
as well as numerous examples can be found in \cite{loc1, loc2}.

\end{document}